\documentclass[12pt]{amsart}
\usepackage{graphicx} 
\usepackage{amsmath,amsfonts,amsthm,amssymb}
\usepackage{enumitem}
\usepackage{dsfont}
\usepackage{xcolor}
\usepackage{tabularx}
\usepackage{hyperref}
\usepackage{todo}
\usepackage[
backend=biber,
style=numeric,
sorting=nty,
maxbibnames=99
]{biblatex}
\usepackage[margin=1in]{geometry}
\setlist[enumerate,1]{label = \arabic*)}

\newtheorem{theorem}{Theorem}
\newtheorem{definition}[theorem]{Definition}
\newtheorem{proposition}[theorem]{Proposition}
\newtheorem{lemma}[theorem]{Lemma}
\newtheorem{corollary}[theorem]{Corollary}
\newtheorem{conjecture}[theorem]{Conjecture}
\newtheorem{question}[theorem]{Question}
\newtheorem{example}[theorem]{Example}
\theoremstyle{remark}
\newtheorem{remark}[theorem]{Remark}
\numberwithin{theorem}{section}

\title{Representation Theory of General Linear Supergroups in Characteristic 2}
\author{Serina Hu}
\address{Massachusetts Institute of Technology, Cambridge, MA, United States}
\email{serinahu@mit.edu}

\newcommand{\mcal}{\mathcal}
\newcommand{\mfrak}{\mathfrak}
\newcommand{\mbb}{\mathbb}
\newcommand{\1}{\mathds{1}}
\newcommand{\op}{\text{op}}
\newcommand{\inj}{\hookrightarrow}
\newcommand{\on}{\operatorname}
\DeclareMathOperator{\Ver}{Ver}
\DeclareMathOperator{\End}{End}
\DeclareMathOperator{\Hom}{Hom}
\DeclareMathOperator{\im}{im}
\DeclareMathOperator{\CommAlg}{CommAlg}
\DeclareMathOperator{\CommHopfAlg}{CommHopfAlg}
\DeclareMathOperator{\ind}{ind}
\DeclareMathOperator{\Grp}{Grp}
\DeclareMathOperator{\AffGrp}{AffGrp}
\DeclareMathOperator{\Spec}{Spec}
\DeclareMathOperator{\id}{id}
\DeclareMathOperator{\Sym}{Sym}

\DeclareMathOperator{\GL}{GL}
\DeclareMathOperator{\PGL}{PGL}
\DeclareMathOperator{\Rep}{Rep}

\DeclareMathOperator{\SL}{SL}
\DeclareMathOperator{\sVec}{sVec}
\DeclareMathOperator{\Dist}{Dist}

\DeclareMathOperator{\diag}{diag}
\DeclareMathOperator{\hw}{hw}

\DeclareMathOperator{\Irrep}{Irrep}
\DeclareMathOperator{\Aut}{Aut}
\DeclareMathOperator{\Mat}{Mat}

\DeclareMathOperator{\Ext}{Ext}
\DeclareMathOperator{\Fr}{Fr}
\DeclareMathOperator{\Aff}{Aff}
\DeclareMathOperator{\colim}{colim}

\newcommand{\comma}{{,}}

\usepackage{breqn}
\newcommand{\breakingplus}{%
  \begingroup\lccode`~=`+
  \lowercase{\endgroup\expandafter\def\expandafter~\expandafter{~\penalty0 }}}
\usepackage{makecell}

\begin{document}

\maketitle

\begin{abstract}
We develop representation theory of general linear groups in the category
$\Ver_4^+$, the simplest tensor category which is not Frobenius exact. Since $\Ver_4^+$
is a reduction of the category of supervector spaces to characteristic $2$
(by a result of Venkatesh), these groups may be viewed as general linear supergroups
in characteristic $2$. More precisely, every object in $\Ver_4^+$ has the form $m\1+nP$ where $P$ is the indecomposable projective, and $\GL(m\1+nP)$ is the reduction to characteristic $2$ of $\GL(m+n|n)$.
We explicitly describe the irreducible representations
of $\GL(P)$ and then use this description to classify the irreducible representations of $\GL(m\1+nP)$ for general $m,n$.
We also define some subgroups of $\GL(m\1+nP)$ and classify their irreducible representations. Finally,
we conjecture a Steinberg tensor product theorem for $\Ver_4^+$ involving the square of the
Frobenius map.
\end{abstract}

\tableofcontents

\section{Introduction}

One of the most important problems in representation theory is understanding the category of polynomial representations of the general linear group $\GL(n)$, as well as the limit as $n\to \infty$ of these categories  -- the category of polynomial functors (\cite{macdonald_symmetric_1995},\cite{friedlander_cohomology_1997}). While over algebraically closed fields of characteristic zero representations of $GL(n)$ are fully understood through the classical work of Schur and Weyl, in positive characteristic the situation is much more complicated, and even computing the dimensions of irreducible representations is very difficult in general. However, the classification of irreducible representations is characteristic free - they are classified by highest weights $(\lambda_1,...,\lambda_n)$ where $\lambda_1\ge...\ge \lambda_n$, $\lambda_i\in \mathbb Z$ (\cite{jantzen_algebraic_groups_2003}). The story is similar for the general linear supergroups $\GL(m|n)$ and the corresponding categories of superpolynomial functors (\cite{axtell_superpolynomial_2013},\cite{drupieski_cohomology_2016}) - the category of representations of $\GL(m|n)$ is rather complicated already in characteristic zero and much more so in positive characteristic, but the classification of irreducible representations by highest weights is available and uniform in all characteristics (\cite{cheng_wang_superalgebras_2012}). 

Recently there has been significant interest in generalizations of this theory to symmetric tensor categories $\mathcal C$ of moderate growth, i.e., studying the representation category of the affine group scheme $\GL(X)$ for $X\in\mathcal C$ and the corresponding category of polynomial functors obtained by taking an appropriate limit with respect to $X$ (see e.g. \cite{venkatesh_representations_2022},\cite{coulembier_inductive_2024}). It is easy to show that if $F: \mathcal C\to \mathcal D$ is a tensor functor then the representation category of $\GL(X)$ in $\mathcal C$ (compatible with the fundamental group of $\mathcal C$) is canonically equivalent to one for $\GL(F(X))$ in $\mathcal D$. So, in characteristic zero, since by Deligne's theorem (\cite{deligne_tensorielles_2002}) there exists $F: \mathcal C\to \sVec$, there is nothing new. However, in positive characteristic there are many new possibilities. For example, one can consider this problem in the Verlinde category $\Ver_p$. In this case, the classification of irreducible representations of $\GL(X)$ was obtained by Venkatesh \cite{venkatesh_representations_2022}. 
As the tensor categories admitting a tensor functor $F: \mathcal C\to \Ver_p$
are precisely the Frobenius exact categories (\cite{coulembier_frobenius_2022}), this takes care of the classification of irreducible representations of $\GL(X)$ (and thereby of simple polynomial functors) in such categories. 

However, there are many tensor categories that are not Frobenius exact, 
for example the categories $\Ver_{p^n}^+$, $\Ver_{p^n}$  defined in \cite{benson_symmetric_2020} in characteristic 2 and in \cite{benson_new_2021},\cite{coulembier_monoidal_2021} in general (namely, $\Ver_{p^n}$ is the abelian envelope of the quotient of the category of tilting modules for $\SL_2(k)$ by the $n$-th Steinberg module, while $\Ver_{p^n}^+$ is its subcategory generated by $\PGL_2(k)$-modules) when $n\ge 2$. Moreover, it is conjectured \cite{benson_new_2021} that any moderate growth tensor category $\mathcal C$ in characteristic $p$ admits a tensor functor $F: \mathcal C\to \Ver_{p^\infty}:= \colim_n \Ver_{p^n}$, so it is conjecturally sufficient to understand these examples. The goal of this paper is to do so in the simplest of them, $\mathcal C=\Ver_4^+$, and more generally to develop Lie theory in this category. 

Specifically, we develop representation theory of general linear groups in the category
$\Ver_4^+$, the tensor category of modules over the Hopf algebra $k[d]/d^2$ with $d$ primitive and $\text{char}(k)=2$, with commutativity defined by the $R$-matrix $R:=1\otimes 1+d\otimes d$ (\cite{venkatesh_hilbert_2016}).\footnote{Note that this category has implicitly appeared much earlier than the referenced papers on Verlinde categories in the context of homotopy theory, see \cite{etingof_p-adic_2020}, Remark 3.4.} Since $\Ver_4^+$ is a reduction of the category of supervector spaces to characteristic $2$ (\cite{venkatesh_hilbert_2016}, Subsection 2.3), these groups may be viewed as general linear supergroups in characteristic $2$, and more generally Lie theory in $\Ver_4^+$ may be regarded as a nontrivial version of superLie theory in characteristic $2$ (the trivial one being ordinary Lie theory, as there are no signs). More precisely, every object in $\Ver_4^+$ has the form $m\1 +nP$ where $P$ is the indecomposable projective, and $\GL(m\1+nP)$ is the reduction to characteristic $2$ of $\GL(m+n|n)$. 
We explicitly describe the irreducible representations
of $\GL(P)$ and then use this description to classify the irreducible representations of $\GL(m\1+nP)$ for general $m,n$. 
We also define some subgroups of $\GL(m\1+nP)$ and classify their irreducible representations. Finally, 
we conjecture a Steinberg tensor product theorem for $\Ver_4^+$ involving the square of the Frobenius map.

We note that there are other nontrivial versions of superLie theory in characteristic $2$, 
for example the one discussed in \cite{bouarroudj_vectorial_2020} (see Subsection 1.2.3) and \cite{bouarroudj_classification_2023} (see Subsection 2.2). In this version, Lie superalgebras are defined as $\mathbb Z/2 \mathbb{Z}$-graded Lie algebras, with an additional structure (the squaring map), whereas in our setting Lie superalgebras have no $\mathbb Z/2 \mathbb{Z}$-grading, and instead have a derivation $d$ such that $d^2=0$. 

Although the category $\Ver_4^+$ is incompressible and developing Lie theory in it is therefore a new problem, it does have a non-symmetric fiber functor to $\text{Vec}$, which allows one to do explicit computations in $\Ver_4^+$ involving vectors, as long as one remembers to apply the $R$-matrix to compensate for the failure of this functor to be symmetric. This approach will be heavily used in the paper.

The organization of the paper is as follows. In Section 2, we introduce the category $\Ver_4^+$ and give a brief overview of commutative algebra and representation theory in symmetric tensor categories. We also explain how $\Ver_4^+$ is a reduction of $\sVec$ to characteristic $2$.

In Section 3, we describe the affine group schemes we will be considering in $\Ver_4^+$, namely $\GL(m\1+nP)$ and some subgroups.

Finally, in Section 4, we classify the irreducible representations of general linear groups in $\Ver_4^+$. Namely, we first show that simple representations of $\GL(P)$ are indexed by integer degree, with two representations in each degree unless the degree is $2 \pmod{4}$, in which there is only one simple representation. Then we prove that simple representations of $\GL(m\1+nP)$ are labeled by pairs of an irreducible representation of $\GL(m)$ and an irreducible representation of $\GL(nP)$, and the latter are, in turn, labeled by $n$-tuples of irreducible $\GL(P)$-representations $L_1, \dots, L_n$ with $\deg L_1 \ge \cdots \ge \deg L_n$. We also describe certain subgroups of $\GL(nP)$ and classify their representations. We explicitly describe the structure and tensor products of irreducible representations of $\GL(P)$ and $\GL(\1 + P)$. Finally, we conjecture a Steinberg tensor product theorem for $\GL(m + nP)$.

\subsection{Acknowledgments}
I am deeply grateful to my advisor, Pavel Etingof, for both suggesting the problems in this paper and providing a huge amount of insightful advice on how to approach both the proofs and general ways of thinking about these problems. I am also grateful to Arun Kannan for our helpful conversations about the Verlinde categories and the Steinberg tensor product theorem, Andrew Snowden and Karthik Ganapathy for asking many excellent questions about $\Ver_4^+$ during a seminar talk on this paper, some of which are answered here, and Kevin Coulembier for our discussions on $\Ver_4^+$ and the Steinberg tensor product theorem.  I am also grateful to the anonymous reviewers who provided feedback on this article. This work was partially supported by NSF grant DMS-2001318.

\section{Preliminaries}

In this section, we introduce the definitions and conventions used in the rest of the paper.

\subsection{Symmetric Tensor Categories}
We refer to \cite{etingof_tensor_2015} for a full introduction to the theory of tensor categories. 
\begin{definition}
We say a category $\mcal{C}$ is a \emph{symmetric tensor category} over a field $k$ if $\mcal{C}$ is
\begin{enumerate}
    \item $k$-linear (Hom spaces are $k$-vector spaces and composition is bilinear);
    \item symmetric monoidal with tensor product $\otimes$ and braiding $c$;
    \item rigid (objects have duals);
    \item Artinian (objects have finite length);
\end{enumerate}
and
\begin{enumerate}[resume*]
    \item $\otimes$ is bilinear on morphisms;
    \item $\End_{\mcal{C}}(\1) = k$.
\end{enumerate}
\end{definition}
(Symmetric) tensor categories are a natural generalization of representation categories for affine group and supergroup schemes. For a more detailed discussion of symmetric tensor categories, we refer the reader to \cite{etingof_lectures_2021}. 

Throughout this paper, we will make little distinction between a symmetric tensor category and its ind-completion; for example, the Hopf algebras we consider will be ind-algebras rather than algebras. 

We refer to \cite{etingof_frobenius_2020} and \cite{coulembier_frobenius_2022} for a discussion on Frobenius functors in general symmetric tensor categories. In characteristic 2, following \cite{etingof_frobenius_2020}, we can define the Frobenius functor as
\begin{equation*}
    \Fr(X) = \ker (1 - c_{X, X}) / \im (1 - c_{X, X}).
\end{equation*}
If $\Fr$ is exact, we say that $\mcal{C}$ is Frobenius exact. In \cite{coulembier_frobenius_2022}, it is proven that a Frobenius exact symmetric tensor category of moderate growth over a field of characteristic 2 is equivalent to the representation category of some affine group scheme.

\subsection{Commutative algebras and (affine) group schemes in tensor categories}
\begin{definition}
    An \emph{associative algebra} $(A, \mu, \eta)$ in a symmetric tensor category $\mcal{C}$ is an object $A \in \ind \mcal{C}$ equipped with a multiplication map $\mu: A \otimes A \to A$ and a unit map $\eta: \1 \to A$ satisfying the standard associativity and unit axioms. A \emph{commutative algebra} is an associative algebra that also satisfies the commutativity axioms for $\mu$. \footnote{See \cite{coulembier_commutative_2023} for an overview of commutative algebra in STCs and \cite{coulembier_algebraic_2023} for an overview of algebraic geometry in symmetric tensor categories.} We denote the category of commutative algebras in $\mcal{C}$ by $\CommAlg \mcal{C}$.

    A \emph{coalgebra} $(H, \Delta, \epsilon)$ in $\mcal{C}$ is likewise an ind-object with homomorphisms $\Delta: H \to H \otimes H$, $\epsilon: H \to \1$ satisfying the standard coassociativity and counit axioms. 

    A Hopf algebra is an ind-object $H$ that is a bialgebra with antipode $S: H \to H$ satisfying the usual axioms for a Hopf algebra.
\end{definition}

In particular, a (unital) subalgebra of $A$ is a subobject $B$ with $\mu(B \otimes B) = B$ and $B$ containing the image of $\eta$, while a right ideal in $A$ is a subobject $I$ such that $\mu(I \otimes A) = I$; we can likewise define left and two-sided ideals. A Hopf subalgebra $B$ of $H$ is a subalgebra of $H$ with $\Delta(B) \subset B \otimes B$, and a Hopf ideal $I$ of $H$ is an ideal with $\Delta(I) \subset H \otimes I + I \otimes H$, $S(I) = I$.

Affine group schemes can be approached via two perspectives, the Grothendieck functor of points and coordinate rings. Throughout this paper, we will use both to study their representations.

\begin{definition}
An \emph{affine group scheme} $G$ in $\mcal{C}$ is a representable functor $\CommAlg \mcal{C} \to \Grp$ (the Grothendieck functor of points).
\end{definition}
In more detail, the coordinate ring $\mcal{O}(G)$ is the (ind-)algebra representing $G$, i.e.
\begin{equation*}
    G(A) = \Hom(\mcal{O}(G), A).
\end{equation*} The morphisms $\mu_G: G \times G \to G$ (multiplication), $i_G: G \to G$ (inversion), and $\eta_G: \1 \to G$ (unit) give $\mcal{O}(G)$ the structure of a Hopf algebra. The comultiplication, counit, and antipode are 
\begin{align*}
\Delta_G &= \mu_G^*: \mcal{O}(G) \to \mcal{O}(G) \otimes \mcal{O}(G), \\
\varepsilon_G &= \eta_G^* : \mcal{O}(G) \to k, \\
\sigma_G &= i_G^*: \mcal{O}(G) \to \mcal{O}(G).
\end{align*}
Conversely, to every commutative Hopf algebra $A$ we can associate an affine group scheme $\Spec A$ given by $(\Spec A)(B) = \Hom(A, B)$. These functors are quasi-inverses and \begin{equation*}\AffGrp \mcal{C} \cong (\CommHopfAlg \mcal{C})^{\op},\end{equation*} where $\AffGrp \mcal{C}$ is the category of affine group schemes in $\mcal{C}$.

\begin{remark}
    Alternatively, an affine group scheme in $\mcal{C}$ is a group object in the category of affine schemes in $\mcal{C}$, i.e. the category whose objects are representable functors $\CommAlg \mcal{C} \to \text{Set}$ and whose morphisms are natural transformations between such functors. 
\end{remark}

\begin{definition}
    An affine group scheme is \emph{finite} if $\mcal{O}(G)$ is an object, not just an ind-object. An affine group scheme is of \emph{finite type} if $\mcal{O}(G)$ is finitely generated.
\end{definition}

\begin{definition}
A \emph{morphism of affine group schemes} $G \to H$ is a natural transformation $f: G \to H$; that is, $f(A): G(A) \to H(A)$ is a group homomorphism for all $A \in \CommAlg \mcal{C}$ commuting with algebra morphisms. Equivalently, it is a morphism of Hopf algebras \newline $\mcal{O}(H) \to \mcal{O}(G)$.
\end{definition}

\begin{example}
    Let $X \in \mcal{C}$ be an object. The general linear group $\GL(X)$ is the affine group scheme such that
    \begin{equation*}
        \GL(X)(A) = \End_A(X \otimes A)^\times.
    \end{equation*}
\end{example}

\begin{definition}
    A \emph{representation of an affine group scheme} $G$ is an object $M \in \mcal{C}$ with a homomorphism $G \to GL(M)$; that is, a group homomorphism $G(A) \to \End(M \otimes A)^\times$ for every $A \in \CommAlg \mcal{C}$ commuting with algebra morphisms. Alternately, it is an $\mcal{O}(G)$-comodule.
\end{definition}

We can take direct sums, tensor products, and duals of representations. Additionally, $G$ has left and right regular representations on $\mcal{O}(G)$: the left regular coaction of $\mcal{O}(G)$ on itself is $(\sigma_G \otimes \id) \circ \Delta_G$, and the right regular coaction is $\Delta_G$. However, note that $\mcal{O}(G)$ is an ind-object. For the rest of the paper, when we refer to an irreducible representation, we mean a finite length one, i.e. not just an ind-object.
\subsubsection{The fundamental group of a symmetric tensor category}

Recall that every symmetric tensor category $\mcal{C}$ has a canonical affine group scheme associated to it, its fundamental group $\pi_1(\mcal{C})$. As a functor of points, $\pi_1(\mcal{C})$ is defined by $\pi_1(\mcal{C})(A) = \Aut^\otimes(A \otimes -)$ (that is, tensor automorphisms of the functor $A \otimes -$). The fundamental group was introduced in \cite{deligne_categories_1990}. $\pi_1(\mcal{C})$ has a natural action on every object $X \in \mcal{C}$ since $\pi_1(\mcal{C})(A)$ acts on $A \otimes X$.

Any symmetric tensor functor $F: \mcal{C} \to \mcal{D}$ induces a map $\epsilon: \pi_1(\mcal{D}) \to F(\pi_1(\mcal{C}))$, and Tannakian reconstruction states that $\mcal{C} \cong \Rep_{\mcal{D}}(F(\pi_1(\mcal{C})), \epsilon)$: that is, the category of representations of $F(\pi_1(\mcal{C}))$ that, when restricted to $\pi_1(\mcal{D})$ via $\epsilon$, give the natural action of $\pi_1(\mcal{D})$. 

In this paper, we are only concerned with representations of affine group schemes that are compatible with the action of the fundamental group of the ambient tensor category.

\subsubsection{Distribution Algebras}

Let $G$ be an affine group scheme of finite type. The Hopf algebra $\mcal{O}(G)$ has the augmentation ideal $I = \ker \epsilon_G$, the kernel of the counit $\epsilon_G$. Since $G$ is of finite type, $\mcal{O}(G)/I^n$ is finite-dimensional. Then the distribution algebra of $G$ is the ind-object
\begin{equation*}
    \Dist(G) = \bigcup_{n = 0}^\infty (\mcal{O}(G)/I^n)^*.
\end{equation*}
When $G$ is finite, $\Dist(G)$ is called the group algebra of $G$ and is also denoted $kG$.

\begin{proposition}
    $\Dist(G)$ is a cocommutative Hopf algebra. Moreover, $(I/I^2)^* \subset \Dist(G)$ is a Lie subalgebra.
\end{proposition}
\begin{proof}
    The proof that $\Dist(G)$ is a cocommutative Hopf algebra is identical to the proof in \cite{venkatesh_harish-chandra_2022} Lemma 4.31. $(I/I^2)^*$ is a Lie subalgebra of $\Dist(G)$ because it is the maximal primitive subobject of $\Dist(G)$, as primitivity is equivalent to vanishing on $1$ and on $I^2$.
\end{proof}
\begin{definition}
    The Lie algebra $\mfrak{g}$ of $G$ is $(I/I^2)^*$.
\end{definition}

\begin{definition}
    Let $H \subset G$ be an inclusion of affine group schemes. Then for an $H$-module $V$, the induced $G$-module is $$\ind_H^G V = \Dist(G) \otimes_{\Dist(H)} V.$$
\end{definition}

\subsection{The category $\Ver_4^+$}

Let $k$ be a field of characteristic 2, not necessarily algebraically closed. The categories $\Ver_{2^n}(k), \Ver_{2^n}^+(k)$ were introduced in \cite{benson_symmetric_2020}. For $n = 2$, $\Ver_4^+(k)$ was known before as the reduction to characteristic $2$ of $\sVec$ (\cite{venkatesh_hilbert_2016}). Also, $\Ver_4^+(k)$ has a concrete description. Consider the Hopf algebra $k[d]/d^2$ where $d$ is primitive; then $\Ver_4^+(k)$ is the monoidal category of finite-dimensional $k[d]/d^2$-modules with symmetric braiding given by $1 \otimes 1 + d \otimes d$, i.e.
\begin{equation*}
    c(v \otimes w) = w \otimes v + dw \otimes dv.
\end{equation*}
From now on, we abbreviate $dx$ to $x'$ for brevity.

The category $\Ver_4^+(k)$ has one simple object $\1$, the trivial representation, with indecomposable projective cover $P=k[d]/d^2$ the regular representation. These are the only indecomposable representations, so every object in $\Ver_4^+(k)$ is of the form $m\1 + nP$. For brevity, we will often denote $m\1$ by $m$. We have $P^* \cong P$. $\Ver_4^+(k)$ is not Frobenius exact since $\Fr(P) = 0$.

Notice that $\Ver_4^+(k)$ admits a forgetful functor to $\text{Vec}_k$, although this functor does not respect the braiding. Nevertheless, this means we can treat objects in $\Ver_4^+$ concretely as vector spaces with additional structure (an automorphism $d$ such that $d^2 = 0$), and any relations involving the braiding can be written in terms of elements of these vector spaces and the $d$-action using the $R$-matrix. Hence, we can speak of elements and bases of objects in $\Ver_4^+(k)$.

For example, in $\Ver_4^+(k)$, a commutative algebra $A$ corresponds to an ordinary, possibly non-commutative algebra over $k$ equipped with a derivation $d$ with $d^2 = 0$. Since $d$ is primitive, 
\begin{equation*}
    (ab)' = a'b + ab'
\end{equation*}
and the commutativity of $A$ in $\Ver_4^+(k)$ is equivalent to
\begin{equation*}
    ab = ba + a'b'.
\end{equation*}
In particular, we obtain that
\begin{equation*}
    (a^n)' = n a^{n-1}a',
\end{equation*}
\begin{equation*}
    (a')^2 = 0,
\end{equation*}
and that $\ker d$ is in the center of $A$.


For the rest of the paper, $k$ will denote an algebraically closed field of characteristic 2 unless otherwise specified, and we will refer to $\mcal{C} = \Ver_4^+(k)$, dropping $k$ from the notation. 

\subsubsection{$\Ver_4^+$ as reduction of $\sVec$ to characteristic 2}

In this section, let $K := \mbb{Q}_2(\sqrt{2})$, where $\mbb{Q}_2$ is the field of $2$-adic numbers. We recall how $\Ver_4^+(\mbb{F}_2)$ can be constructed as a non-semisimple reduction of $\sVec(K)$ over $K$. This construction was described in \cite{venkatesh_hilbert_2016}, Section 2.3. Let $H := K[\mbb{Z}/2\mbb{Z}]$ be the group algebra of $\mbb{Z}/2\mbb{Z}$ as a Hopf algebra, i.e. it is generated by $1, g$ such that $g^2 = 1$ and $\Delta(g) = g \otimes g$.
Recall that in characteristic not 2, the category of supervector spaces is the $H$-modules $\Rep(H)$ with associated $R$-matrix
\begin{equation*}
    \frac{1}{2}(1 \otimes 1 + 1 \otimes g + g \otimes 1 - g \otimes g).
\end{equation*}
Setting $b = 1 - g$, we can rewrite
\begin{equation*}
    R = 1 \otimes 1 - \frac{1}{2} b \otimes b
\end{equation*}
and setting $d = \frac{1}{\sqrt{2}}b$, we have
\begin{equation*}
    R = 1 \otimes 1 - d \otimes d.
\end{equation*}
Additionally, $d^2 = \sqrt{2}d$ and
\begin{equation*}
    \Delta(d) = 1 \otimes d + d \otimes 1 + \sqrt{2}(d \otimes d).
\end{equation*}

Therefore, $d$ generates an order in $H$ over $\mcal{O} := \mathbb Z_2[\sqrt{2}]$. Reducing this order modulo the maximal ideal in $\mcal{O}$, which is the ideal generated by $\sqrt{2}$, we get the Hopf algebra $\mathbb{F}_2[d]/d^2$ over $\mathbb{F}_2$  with $a$ primitive and braiding $(1 \otimes 1 + d \otimes d) \circ \sigma$. Thus we obtain the symmetric tensor category $\Ver_4^+(\mbb{F}_2)$. For other fields $K$, $\Ver_4^+(K)$ is obtained via extension of scalars as
\begin{equation*}
    \Ver_4^+(K) := \Ver_4^+(\mathbb{F}_2) \otimes_{\mbb{F}_2} K.
\end{equation*}

\begin{proposition}
    Modulo the maximal ideal of $\mcal{O}$, the supervector space $K^{(1, 0)}$ with one-dimensional even part reduces to $\1 \in \Ver_4^+(\mbb{F}_2)$, whereas the supervector space $K^{(1, 1)}$ with one-dimensional even part and one-dimensional odd part reduces to $P \in \Ver_4^+(\mbb{F}_2)$.
\end{proposition}
\begin{proof}
    If $v$ is even, $gv = v$, and if $w$ is odd, $gw = -w$. Since $d = \frac{1 -g}{\sqrt{2}}$, we have $dv = 0$ and $dw = \sqrt{2}w$. Therefore, for $V = k^{(1, 0)} = \langle v \rangle$, its reduction modulo $\sqrt{2}$ is $\1$.

    For $V = k^{(1, 1)} = \langle v, w \rangle$, consider the (inhomogeneous) basis $p = v + w$ and $q = \sqrt{2}w$. Then $dp = q$ and $dq = \sqrt{2}q$. Reducing modulo $\sqrt{2}$, we get $dp = q, dq = 0$, giving $P$.
\end{proof}
\begin{corollary}
    The supervector space $K^{m + n|n}$ reduces to $m + nP \in \Ver_4^+(\mbb{F}_2)$.
\end{corollary}

\subsubsection{The fundamental group of $\Ver_4^+$}

\begin{proposition}
    The fundamental group $\pi 
 = \underline{Aut}^{\otimes}(\on{Id})$ of $\Ver_4^+$ has group algebra \begin{equation*}k\pi = k[d]/d^2,\end{equation*} the commutative and cocommutative Hopf algebra in $\Ver_4^+$ which is contained in $\on{Vec}$. As a functor of points, it sends $A$ to elements that square to 0 in $\ker d \subset A$ under addition.
\end{proposition}
\begin{proof}
    Recall that $\pi(A) = \underline{\Aut}^{\otimes}(- \otimes A)$ and let $\eta \in \pi(A)$ be one such tensor automorphism. Then $\eta_{\1} \in \Aut(A)$ is multiplication by an invertible element in $A$, and $\eta_{\1 \otimes \1} = \eta_{\1} \otimes \eta_{\1}$ implies that $\eta_{\1}$ is the identity map. This also determines $\eta_{m\1}$, since $\eta$ is additive.

    It then suffices to determine $\eta_{P} \in \GL(P)(A)$ and ensure $\eta_{P \otimes P}$ respects isomorphisms $P \otimes P \cong 2P$. Writing $\eta_{P}$ as a matrix $\begin{pmatrix} a & a' \\ b & a + b' \end{pmatrix}$ with entries in $A$, we know that $\eta_{P}$ must be trivial on $\im d$, so $a' = 0, a + b' = 1$. Additionally, since $\eta$ is additive, $\eta_{nP}$ fixes $n\1 \subset nP$. For $x \in nP$ not in $\ker d$, $\eta(x \otimes 1) = x \otimes (1 + b') + x' \otimes b$.
    
    Suppose we have an isomorphism $P \otimes P \cong 2P$ sending $1 \otimes 1 \mapsto x$, $d \otimes 1 \mapsto y$, so $1 \otimes d + d \otimes 1 \mapsto x'$ and $d \otimes d \mapsto y'$. Since $\eta_{P \otimes P}$ must fix elements in $\im d$, in particular $1 \otimes d + d \otimes 1$, we must have $1 + b' = 1$, so $b' = 0$. Also, we want
    \begin{equation*}
        \eta_{P \otimes P}(1 \otimes 1) = (1 \otimes 1) \otimes (1 + b') + (d \otimes 1 + 1 \otimes d) \otimes b
    \end{equation*}
    and we can check this requires $b^2 = 0$. Then $$\eta_{P \otimes P}(d \otimes 1) = (d \otimes 1) \otimes (1 + b') + (d \otimes d) \otimes b,$$ and likewise for $1 \otimes d$. Therefore, $\eta_P = \begin{pmatrix} 1 & 0 \\ b & 1 \end{pmatrix}$ where $b^2 = 0$, $b' = 0$, so $\eta$ is parametrized by elements in $\ker d$ that square to $0$.
    
    Therefore, $k\pi = k[d]/d^2$.
\end{proof}

\subsection{Highest Weights} 
\label{highest_weights}

From this section onwards, we assume our symmetric tensor category $\mcal{C}$ satisfies hypothesis 1.3.1 in \cite{coulembier_algebraic_2023}, which allows us to develop scheme theory in $\mcal{C}$. In particular, $\Ver_4^+$ satisfies this hypothesis, see \cite{coulembier_algebraic_2023} Section 8.

Let $G$ be an affine group scheme of finite type in such a symmetric tensor category $\mcal{C}$ with $\phi: \pi_1 (\mcal{C}) \to G$. Let $f: \mbb{G}_m \to G$ be a homomorphism which lands in the centralizer of $\phi(\pi_1(\mcal{C}))$. Let $Z(f)$ be the centralizer of the image of $f$ in $G$. Then for every irreducible representation $L$ of $G$ compatible with $\phi$, we have the representation $\hw_f(L)$ of $Z(f)$ compatible with $\phi$ on the eigenobject of $\mbb{G}_m$ on $L$ (acting via $f$) with largest eigenvalue (which is an integer). We call $\hw_f(L)$ the highest weight of $L$ with respect to $f$.

\begin{theorem}\label{highest_weight_thm}
    \begin{enumerate}
        \item $\hw_f(L)$ is irreducible.
        \item The assignment $L\mapsto \hw_f(L)$ is an injective map $\hw_f: \Irrep(G,\phi) \to \Irrep(Z(f),\phi)$.
        \item If $X,Y$ belong to the image $I(f)$ of $\hw_f$ then every composition factor of $X\otimes Y$ also belongs to $I(f)$.
    \end{enumerate}
\end{theorem}

\begin{proof}
    Let $X$ be a faithful $G$-representation, giving a map $G \hookrightarrow \GL(X)$. Then $f$ induces a $\mathbb G_m$-action on $X$, so we can decompose $X$ into $\mathbb G_m$-eigenobjects as $X = \bigoplus_{i = 1}^n X_i$ so that $t \in \mathbb G_m$ acts as $\diag(t^{m_1}, \cdots, t^{m_n})$. Therefore $\text{GL}(X)$ can be thought of as the group of invertible $n \times n$ matrices where the $ij$th matrix entry is in $X_i \otimes X_j^*$. 

    Let $Z$ be the subgroup of such diagonal matrices (i.e. the only nonzero entries are in $X_i \otimes X_i^*$), $N_+$ be the subgroup of upper triangular matrices with 1s on the diagonal, and $N_-$ be the subgroup of lower triangular matrices with 1s on the diagonal. Then $Z = Z(f)$, and the multiplication map $N_- \times Z \times N_+ \to G$ is an open immersion of schemes. Hence this is an isomorphism on a formal neighborhood of the identity and we get a ``PBW factorization'' for $\Dist(G)$ as 
    \begin{equation} \label{eq:pbw}
        \Dist(G) \cong \Dist(N_-) \otimes \Dist(Z) \otimes \Dist(N_+)
    \end{equation}
    as objects: the multiplication map from the right-hand side to the left-hand side is an isomorphism.

    $\Dist(G)$ also has a weight decomposition with respect to the adjoint action of $\im f$. In this grading, $\Dist(Z)$ has weight 0. The augmentation ideal of $\Dist(N_+)$ is a direct sum of strictly positive weights, and likewise the augmentation ideal of $\Dist(N_-)$ is a direct sum of strictly negative weights. 

    Then statement 1 follows from this PBW factorization \eqref{eq:pbw}: if $M \subset \hw_f(L)$ is a proper $\Dist(Z)$ submodule, then the $G$-module generated by $M$ will be a proper submodule of $L$, since $N_+$ acts trivially on $M$ while $Z$ fixes $M$ and $N_-$ lowers the $\mathbb G_m$-eigenvalue, i.e. it acts trivially on the quotient by lower weights.
    

   Set $B_+ = Z \times N_+$; then the multiplication map is an isomorphism of ind-objects 
   \begin{equation*}
       \Dist(B_+) \cong \Dist(Z) \otimes \Dist(N_+).
   \end{equation*}
   Given a highest weight $M$, we can therefore consider the Verma module $\Dist(G) \otimes_{\Dist(B_+)} M$. Then the map $L \mapsto \hw_f(L)$ is injective because this Verma module has the usual universal property, and has a unique maximal submodule, hence unique irreducible quotient when $M$ is an irreducible $Z$-representation. This proves statement 2.
    
    For statement 3, suppose that $L(X)$ and $L(Y)$ are the simple $G$-representations with highest weights $X, Y$ respectively. Then $L(X) \otimes L(Y)$ has highest weight $X \otimes Y$, and by the same argument in the proof of statement 1, the composition factors of $L(X) \otimes L(Y)$ will have highest weights the respective composition factors of $X \otimes Y$.
\end{proof}

\begin{corollary}
    Because irreducible $B_+$-modules are in bijection with irreducible $Z$-modules, we can also interpret $L \mapsto \hw_f(L)$ as an injective map $\hw_f: \Irrep(G, \phi) \to \Irrep(B_+, \phi)$.
\end{corollary}

\begin{proposition}
An irreducible $B_+$-module $V$ lies in the image of $\hw_f$ iff $\Hom_{B_+}(V, \mcal{O}(G))$ is nonzero.
\end{proposition}
\begin{proof}
    Suppose there exists some nonzero $h \in \Hom_{B_+}(V, \mcal{O}(G))$. Let $M$ be the $G$-module generated by $h$ in $\mcal{O}(G)$. It will have highest weight $V$, since the $G$-action can only lower the $\mbb{G}_m$-eigenvalue by the same argument as in the proof of Theorem \ref{highest_weight_thm}. Because $M$ is a subobject of $\mcal{O}(G)$, which is an ind-object, $M$ will be an honest object in $\mcal{C}$. Therefore, $V$ lies in the image of $\hw_f$.

    Conversely, suppose that $V$ lies in the image of $\hw_f$, so there exists some irreducible $G$-module $M$ with highest weight $V$. Then the matrix coefficients of $M_V \subset M$, the subobject of $M$ with weight $V$, will provide the $B_+$-invariant map $V \to \mcal{O}(G)$.
\end{proof}

\section{Examples} \label{examples}

In this section, we give examples of commutative algebras and affine group schemes in $\Ver_4^+$ that we will consider in the rest of the paper.

\subsection{Symmetric algebras}

Every object $X$ of a symmetric tensor category is canonically an abelian affine group scheme. Its underlying Lie algebra is the abelian Lie algebra with underlying object $X$, and its coordinate ring is $\Sym(X^*)$.

For example, if $X = P \cong P^*$, the algebra $\Sym P$ is the polynomial algebra $k[x, x']$ (that is, as an ordinary algebra it is $k[x,y]/y^2$ and the $d$-action is $y = x'$). $\Sym^n P$ has basis $x^n, x^{n-1}x'$, so
\begin{equation*}
    \Sym^n P \cong \begin{cases}
        P & n \text{ odd} \\
        \1 \oplus \1 & n \text{ even}.
    \end{cases}
\end{equation*}

\subsection{$\mathbb G_a, \mathbb G_m$, and friends}

Classically, the additive group $\mathbb G_a$ has coordinate algebra $k[T]$ and is the functor sending an algebra $A$ to the group $(A, +)$. Likewise, the multiplicative group $\mathbb G_m$ has coordinate algebra $k[T, T^{-1}]$ and sends $A$ to $A^\times$. Note that in $\Ver_4^+$, the functor represented by $k[T]$ with $dT = 0$ sends $A \mapsto (\ker d \in A, +)$, while the functor $A \mapsto (A, +)$ is represented by $k[T, T'] = \Sym P$, so these are different groups.

\begin{definition}
Let $\mathbb G_a$ refer to the affine group scheme taking $A \mapsto (\ker d \in A, +)$. Let $\mathbb G'_a$ refer to the affine group scheme taking $A \mapsto (A, +)$. 
\end{definition}
The coordinate ring of $\mathbb G_a$ is $k[T]$. The coordinate ring of $\mathbb G'_a$ is $k[T, T'] = \Sym P$, $T$ and $T'$ primitive, and its Lie algebra is abelian with underlying object $P$, so it is the abelian group scheme associated with $P$.
\begin{definition}
Let $\mathbb G_m$ refer to the affine group scheme with $\mathbb G_m(A) = \ker d |_{A^\times}$, i.e. invertible elements in the kernel of $d$. Let $\mathbb G'_m$ refer to the affine group scheme taking $A$ to $A^\times$.
\end{definition}
The coordinate ring of $\mathbb G_m$ is $k[T, T^{-1}]$, $T$ grouplike. The coordinate ring of $\mathbb G'_m$ is $k[T, T^{-1}, T']$, $T$ grouplike. Its Lie algebra's underlying object is again $P$, but it has a nonzero bracket, $[x, x] = x'$. 
\begin{remark}
    Note that $A$ is a commutative algebra in $\mcal{C}$, but not in the usual sense in general, so $A^\times$ is a non-abelian abstract group. Hence the group scheme $\mathbb G'_m$ is non-abelian. 
\end{remark}

\begin{definition}
Let $M_1$ denote the affine group scheme taking $A$ to $A$ with group operation $a * b = a + b + a'b$.
\end{definition}
$M_1$ has coordinate ring $k[X, X']$, $\Delta(X) = X \otimes 1 + 1 \otimes X + X' \otimes X$, $S(X) = X + XX'$. Therefore, $1 + X'$ is grouplike. The Lie algebra of this group has underlying object $P$, with nonzero brackets $[x, x'] = [x', x] = x'$.

Since $\Ver_4^+$ is a reduction of $\sVec$ to characteristic 2, let us compare this to the corresponding groups in characteristic $0$ in $\sVec$. For the rest of the subsection, we let $k = \mbb{Q}_2(\sqrt{2})$.
The following two propositions are well-known statements about low-dimensional Lie superalgebras and supergroups.



\begin{proposition}
Let $\mfrak{g} = L$ as an abelian Lie algebra. Consider the supergroup $G_a^{1|1}$ with $\mathcal{O}(G) = k[X, \Xi]$, $X$, $\Xi$ both primitive, $X$ even and $\Xi$ odd, which has Lie algebra $\mfrak{g}$ and sends a supercommutative algebra $A$ to $(A, +)$. This supergroup degenerates to $\mathbb G'_a$ in $\Ver_4^+$.
\end{proposition}
\begin{proof}
    Let $k = \mbb{Q}_2[\sqrt{2}]$ and $k[\mathbb Z / 2 \mathbb Z] = \langle 1, 
g \rangle$.
    Since $X$ is even and $\Xi$ is odd, $gX = X$, $g\Xi = -\Xi$. Moreover, $d = \frac{1 - g}{\sqrt{2}}$, so $dX = 0$ and $d\Xi = \sqrt{2}\Xi$. Consider $p = X + \Xi$ and $q = \sqrt{2}\Xi$, which also generate $\mathcal{O}(G)$. Then $dp = q, dq = \sqrt{2}q$, and $p, q$ are both primitive.
    
    Reducing modulo $\sqrt{2}$, we get the commutative Hopf algebra $k[p, q]$ with $dp = q, dq = 0$, and $p, q$ primitive, which is the coordinate ring of $\mathbb G'_a$.
\end{proof}
\begin{proposition}
    Consider the affine supergroup $\Aff(0, 1)$ of affine transformations of $k^{0|1}$. Its coordinate ring is the commutative Hopf superalgebra over $k$ generated by $Y$ (even, invertible, grouplike) and $Z$ odd with $\Delta(Z) = Z \otimes 1 + Y \otimes Z$. Then the lattice generated by $X := Z + \frac{Y - 1}{\sqrt{2}}$ and $\Xi = \sqrt{2}Z$ degenerates to $M_1$.
\end{proposition}
\begin{proof}
    We can check that $dY = 0$ and $dZ = \sqrt{2}Z$, so $dX = \Xi$ and $d\Xi = \sqrt{2}\Xi$. We compute that
    \begin{align*}
        \Delta(X) &= Z\otimes 1+Y\otimes Z+(Y-1)\otimes \frac{1}{\sqrt{2}}+1\otimes \frac{Y-1}{\sqrt{2}}+(Y-1)\otimes \frac{Y-1}{\sqrt{2}} \\
        &= X\otimes 1+1\otimes X+(Y-1)\otimes Z+(Y-1)\otimes \frac{Y-1}{\sqrt{2}} \\
        &= X\otimes 1+1\otimes X+X\otimes \Xi+\sqrt{2}X\otimes X-\Xi\otimes X-X\otimes \Xi \\
        &= X\otimes 1+1\otimes X+\sqrt{2}X\otimes X-\Xi\otimes X
    \end{align*}
    and
    \begin{align*}
        \Delta(\Xi) &= \sqrt{2}Z\otimes 1+\sqrt{2}Y\otimes Z \\
        &=\Xi \otimes 1+1\otimes \Xi+\sqrt{2}X\otimes \Xi-\Xi\otimes X'.
    \end{align*}
    Therefore, on reducing the lattice generated by $X, \Xi$ modulo $\sqrt{2}$, we get $M_1$. 
\end{proof}
\begin{remark}
    The Lie algebra structure on $k^{1|1}$ with $x$ even and $\xi$ odd and nonzero bracket $[x, \xi] = \xi$ reduces to the Lie algebra of $M_1$. Consider $v = x + \xi$ and $w = \sqrt{2}\xi$. Then $dv = w$ and $dw = \sqrt{2}w$. Moreover, $[v, v] = 0$, $[v, w] = w$, and $[w, w] = 0$. Reducing modulo $\sqrt{2}$, we have $dv = w$, $dw = 0$, and $[v, w] = w$ is the only nonzero bracket, which is $\mathfrak{aff}(0, 1)$, the Lie algebra of the group $\Aff(0, 1)$ of affine transformations of $k^{(0, 1)}$.
\end{remark}


\begin{proposition}
    The $(1|1)$-dimensional supersymmetry group $S(1|1)$ taking a supercommutative algebra $A$ to $A^\times$ in characteristic $0$ degenerates to $\mathbb G'_m$ in characteristic 2.
\end{proposition}
\begin{proof}
Let $k = \mbb{Q}_2[\sqrt{2}]$ and $k[\mathbb Z / 2 \mathbb Z] = \langle 1, 
g \rangle$. Let $A$ be a supercommutative algebra. Every invertible element of $A$ can be written uniquely as $a + \alpha$, $a$ even and invertible, $\alpha$ odd (such an element has inverse $a - \frac{\alpha}{a^2}$). Thus, $\mcal{O}(G) = k[X, Y, X^{-1}]$, where $X$ is even, $Y$ is odd, and 
\begin{align*}
    \Delta(X) = X \otimes X + Y \otimes Y \\
    \Delta(Y) = Y \otimes X + X \otimes Y.
\end{align*}

    Since $X$ is even and $Y$ is odd, $gX = X, gY = -Y$. Moreover, $d = \frac{1 - g}{\sqrt{2}}$, so $dX = 0$ and $dY = \sqrt{2}Y$. Let $p = X + Y$ and $q = \sqrt{2}Y$, so
    \begin{align*}
        \Delta(p) &= p \otimes p, \\
        \Delta(q) &= q \otimes p + p \otimes q - \sqrt{2}q \otimes q,
    \end{align*}
    while $dp = q, dq = \sqrt{2}q$.

    Then reducing the lattice generated by $p, q$ modulo $\sqrt{2}$, we get $k[p, q]$ with $dp = q$, $dq = 0$, $\Delta(p) = p \otimes p$, $\Delta(q) = q \otimes p + p \otimes q$. This is $\mcal{O}(\mathbb G'_m)$.
\end{proof}

\begin{remark}
    The Lie superalgebra of this group has only one nonzero bracket, $[\mu_Y, \mu_Y]$, where $\mu_Y$ is the dual element to $Y$. This Lie superalgebra is $\mfrak{s}(1|1)$, the Lie algebra of $S(1|1)$.
\end{remark}

\subsection{General linear groups and subgroups}
For $X \in \Ver_4^+$, the general linear group $\GL(X)$ is the group sending $A$ to  $\End_A(X \otimes A)^\times$. Suppose $X \cong m + nP$. Then $\GL(m+nP)$ can be thought of as the group of invertible (block) matrices of the form
\begin{equation*}
    \begin{pmatrix}
        m \1 \otimes (m \1)^* & m \1 \otimes (nP)^* \\
        nP \otimes (m \1)^* & nP \otimes (nP)^*
    \end{pmatrix}.
\end{equation*}

Concretely, say $m\1$ has basis $y_1, \dots, y_m$ and $nP$ has basis $x_1, \dots, x_n, x'_1, \dots, x'_n$. Then $A$-automorphisms $(m+nP) \otimes A \to (m+nP) \otimes A$ are determined by the images of $y_i \otimes 1$ and $x_i \otimes 1$, since $d(x_i \otimes 1) = x'_i \otimes 1$. 
\begin{proposition}
As matrices, the elements of $\GL(m+nP)(A)$ can be written as invertible $(m+2n) \times (m+2n)$ matrices of the form
\begin{equation*}
    \begin{pmatrix}
        F & C & C' \\
        B' & D & D' \\
        B & E & D+E'
    \end{pmatrix}
\end{equation*}
where $F \in \End(m\1)(A)$ (so $F' = 0$), $D, E \in \Mat_{n \times n}(A)$, $C \in \Mat_{m \times n}(A)$, $B \in \Mat_{n \times m}(A)$, and the group operation is multiplication of these matrices in the ordinary sense.
\end{proposition}

\begin{corollary}
    Such a block matrix is invertible iff $F$ and $D$ are invertible.
\end{corollary}
\begin{proof}
Because this matrix must be invertible, and invertibility is not affected by nilpotents, invertibility of this block matrix is equivalent to invertibility of
\begin{equation*}
    \begin{pmatrix}
        F & C & 0 \\
        0 & D & 0 \\
        B & E & D
    \end{pmatrix}
\end{equation*}
which is invertible iff $FD^2$ is invertible. Hence, both $F$ and $D$ must be invertible, and this is sufficient.
\end{proof}
\begin{corollary}\label{glm+np_lu}
Let $B_{m, n}$ be the parabolic subgroup sending $A$ to the group of block upper triangular matrices 
\begin{equation}
\label{b_mn_format}
    \begin{pmatrix}
        F & C & C' \\
        0 & D & D' \\
        0 & E & D+E'
    \end{pmatrix}
\end{equation}
with $F, D, E$ as above with entries in $A$. Identify $(\mathbb G'_a)^{mn}$ with $\Mat_{m \times n}(A)$.
Then the multiplication map $(\mathbb G'_a)^{mn} \times B_{m,n} \to \GL(m + nP)$ defined by 
\begin{equation*}
    B, (F, C, D, E) \mapsto \begin{pmatrix}
        I_m & 0 & 0 \\
        B' & I_n & 0 \\
        B & 0 & I_n
    \end{pmatrix}
    \begin{pmatrix}
        F & C & C' \\
        0 & D & D' \\
        0 & E & D+E'
    \end{pmatrix}
\end{equation*}
is an isomorphism of schemes.
\end{corollary}
\begin{proof}
    Because $F$ and $D$ are invertible, each block matrix of the form in \eqref{b_mn_format} can be written uniquely as a product
\begin{equation*}
    \begin{pmatrix}
        I_m & 0 & 0 \\
        B'F^{-1} & I_n & 0 \\
        BF^{-1} & 0 & I_n
    \end{pmatrix}
    \begin{pmatrix}
        F & C & C' \\
        0 & B'F^{-1}C + D & (B'F^{-1}C + D)' \\
        0 & BF^{-1}C + E & B'F^{-1}C + D + (BF^{-1}C + E)'
    \end{pmatrix}.
\end{equation*}
\end{proof}

\begin{proposition}
    Let $M_n$ be the affine group scheme where $M_n(A)$ is the group of $n \times n$ matrices with entries in $A$ with group operation $X * Y = X + Y + X'Y$. Let $H_n$ be the affine group scheme where $H_n(A)$ is the group of $n \times n$ invertible matrices with entries in $A$ under matrix multiplication. The multiplication map $M_n \times H_n \to \GL(nP)$ defined by
    \begin{equation*}
        C, D \mapsto \begin{pmatrix}
        I_n & 0 \\
        C & I_n + C'
    \end{pmatrix}
    \begin{pmatrix}
        D & D' \\
        0 & D
    \end{pmatrix}
    \end{equation*}
    is an isomorphism of schemes.
\end{proposition}
\begin{proof}
A matrix $\begin{pmatrix}
        D & D' \\
        E & D + E'
    \end{pmatrix}$ in $\GL(nP)(A)$ has $D$ invertible. Hence, we can write it uniquely as a product
\begin{equation*}
    \begin{pmatrix}
        D & D' \\
        E & D + E'
    \end{pmatrix}
    =
    \begin{pmatrix}
        I_n & 0 \\
        ED^{-1} & I_n + (ED^{-1})'
    \end{pmatrix}
    \begin{pmatrix}
        D & D' \\
        0 & D
    \end{pmatrix}.
\end{equation*}
\end{proof}

\begin{remark}
The Lie algebra of $\GL(m + nP)$ is
\begin{equation*}
\mfrak{gl}(m+nP) = \underline{\End}(m + nP) = (m+nP) \otimes (m+nP)^*,
\end{equation*}
where $\underline{\End}$ is an internal Hom. Hence $\mfrak{gl}(m+nP)$ is an $(m + 2n)^2$-dimensional vector space, and $d$ acts via $1 \otimes d + d \otimes 1$. The bracket operation is given by $[\phi_1, \phi_2] = \phi_1\phi_2 - \phi_2\phi_1 + \phi_2'\phi_1'$.
\end{remark}

\section{Representations of general linear groups in $\Ver_4^+$}

In this section, we describe the irreducible representations of the affine group schemes defined in Section \ref{examples}. For $\GL(P)$ and $\GL(1+P)$ we also explicitly describe the tensor products of the irreducible representations.

\subsection{Representations of $\mathbb G'_a$ and $\mathbb G'_m$}

\begin{proposition}
    $\mathbb G'_a$ has a unique irreducible representation – the trivial representation.
\end{proposition}
\begin{proof}
Recall that the coordinate ring of $\mathbb G'_a$ is $\mathcal{O}(\mathbb G'_a) = k[T, T']$. $\mathbb G_a$ is a subgroup of $\mathbb G'_a$, so every irreducible representation of $\mathbb G'_a$ lies inside $\ind_{\mathbb G_a}^{\mathbb G'_a} \1 = \mathcal{O}(\mathbb G'_a)^{\mathbb G_a}$, where $\mathbb G_a$ acts via the right regular representation and $\mathbb G'_a$ acts via the left regular representation. An element $a \in \mathbb G_a(A)$ acts on $\mathcal{O}(\mathbb G'_a) \otimes A$ by $T \otimes 1 \mapsto T \otimes 1 + 1 \otimes a$. So $\mathcal{O}(\mathbb G'_a)^{\mathbb G_a} = k[T']$, which is 2-dimensional because $(T')^2 = 0$. The $\mathbb G'_a$-action of $a \in A$ is $T' \otimes 1 \mapsto T' \otimes 1 + 1 \otimes a'$. Hence $k[T']$ is an extension of the trivial representation (generated by $1$) by itself.
\end{proof}

\begin{proposition}
    The irreducible representations $L_n$ of $\mathbb G'_m$ are labeled by integer $n$. Moreover, $L_n$ has the following structure:
    \begin{itemize}
    \item if $n \equiv 0 \pmod{4}$, $L_n$ is one-dimensional defined by $a \mapsto a^n$;
    \item if $n \equiv 1 \pmod{4}$, $L_n$ is a representation on $P$ defined by $a \mapsto \begin{pmatrix} a^n & a^{n - 1}a' \\ 0 & a^n \end{pmatrix}$ with respect to the basis $x, x'$;
    \item if $n \equiv 2 \pmod{4}$, $L_n$ is one-dimensional defined by $a \mapsto a^n + a^{n - 1}a'$;
    \item if $n \equiv 3 \pmod{4}$, $L_n$ is a representation on $P$ defined by $a \mapsto \begin{pmatrix} a^n + a^{n -1}a' & a^{n - 1}a' \\ 0 & a^n + a^{n-1}a' \end{pmatrix}$ with respect to the basis $x, x'$;
\end{itemize}
for $a \in \mathbb G'_m(A) = A^\times$.
\end{proposition}
\begin{remark}
Note that when $n \equiv 1, 3 \pmod{4}$, even though the matrices are upper triangular, $L_n$ is irreducible because it must be closed under the $d$-action. Even though the $\mathbb G'_m$-action fixes $x$, the $d$-action requires that $x'$ also lie in the representation.
\end{remark}
\begin{proof}
The coordinate ring of $\mathbb G'_m$ is $\mathcal{O}(\mathbb G'_m) = k[X, X', X^{-1}]$ with $\Delta(X) = X \otimes X$. For $a \in \mathbb G'_m(A)$, the $\mathbb G'_m$-action on $\mathcal{O}(\mathbb G'_m)$ sends $X$ to $X \otimes a$.

    We have $\mathbb G_m \subset \mathbb G'_m$, so the representations of $\mathbb G'_m$ are found in $\ind_{\mathbb G_m}^{\mathbb G'_m} V = (V \otimes \mathcal{O}(\mathbb G'_m))^{\mathbb G_m}$ where $V$ is a $\mathbb G_m$-representation, $\mathbb G_m$ acts on the left on $V$ and on the right on $\mathcal{O}(\mathbb G'_m)$, and $\mathbb G'_m$ acts on the left on $\mathcal{O}(\mathbb G'_m)$. Let $V_n$ be the $\mathbb{G}_m$-representation of degree $n$. The $\mathbb G_m$-invariants of $V_n \otimes \mathcal{O}(\mathbb G'_m)$ are the degree $n$ polynomials in $X, X^{-1}, X'$. 

    So $(V \otimes \mathcal{O}(\mathbb G'_m))^{\mathbb G_m}$ is spanned by $v \otimes X^n, v \otimes X^{n-1}X'$. If $n$ is even, this is a copy of $\1^2$. We can check that this is an extension of $\1$ by itself and the $\mathbb G'_m$-action on $\1$ is given as above. If $n$ is odd, this is a copy of $P$, and the $\mathbb G'_m$-action is given as above.

\end{proof}

\begin{proposition}
    $M_1$ has two irreducible representations -- the trivial representation and the one defined by $a \mapsto 1 + a'$, which we will denote by $L_\xi$. Both are one-dimensional.
\end{proposition}
\begin{proof}
    $\mathcal{O}(M_1) = k[X, X']$ with coproduct $\Delta(X) = 1 \otimes X + X \otimes 1 + X' \otimes X$. Then $1 + X'$ is grouplike and $a \mapsto 1 + a'$ is a 1-dimensional representation of $M_1$.

    To see that this is the only nontrivial irreducible representation, note that $\mathbb G_a \subset M_1$, and thus all the irreducible representations of $M_1$ are contained in $\ind_{G_a}^{M_1} \1 = (\1 \otimes \mcal{O}(M_1))^{\mathbb G_a}$. The $\mathbb G_a$-action on $\mcal{O}(M_1)$ takes $X$ to $X \otimes 1 + 1 \otimes a$ and $X'$ to $X' \otimes 1$, so
    \begin{equation*}
        (\1 \otimes \mcal{O}(M_1))^{\mathbb G_a} = \1 \otimes k[X']
    \end{equation*}
    and from the description of $\mcal{O}(M_1)$ above, we conclude that the trivial representation and the representation taking $a \mapsto 1 + a'$ are the only irreducible representations of $M_1$.
\end{proof}

\subsection{Representations of $\GL(m + nP)$}

In this section, we describe the finite-dimensional irreducible modules for $\GL(m+ nP)$.

\begin{proposition}\label{glm+np_irrep}
    The irreducible representations of $\GL(m+nP)$ are labeled by pairs of an irreducible representation of $\GL(m)$ and an irreducible representation of $\GL(nP)$.
\end{proposition}
\begin{proof}

By Corollary \ref{glm+np_lu}, $G = \GL(m+nP)$ has a Gauss decomposition $(\mathbb G'_a)^{mn} \times B_{m,n}$. For ease of notation, let $B := B_{m, n}$ and $M := (\mathbb G'_a)^{mn}$. The irreducible $B$-modules are then in bijection with pairs of an irreducible $\GL(m)$-module and an irreducible $\GL(nP)$-module.

Via the discussion in Section \ref{highest_weights}, every irreducible $G$-module has a highest weight, which is the corresponding irreducible $B$-module, and this map from $G$-irreps to $B$-irreps is injective. For any irreducible $B$-module $V$, we have $\Hom_B(V, \mathcal{O}(G)) \ne 0$ iff there exists a $G$-module with highest weight $V$. The Gauss decomposition implies that the multiplication map is an isomorphism $\mathcal{O}(M) \otimes \mathcal{O}(B) \cong \mcal{O}(G)$. Therefore, 
\begin{equation*}
    \Hom_B(V, \mathcal{O}(G)) \cong V^* \otimes \mcal{O}(M)
\end{equation*}
and every irreducible $B$-module is the highest weight for some irreducible $G$-module. We thus get a bijection between irreducible $G$-modules and irreducible $B$-modules.
\end{proof}

Thus, to describe irreducible $\GL(m + nP)$-modules, it suffices to describe irreducible $\GL(nP)$ modules. We first find all finite-dimensional irreducible $\GL(P)$-modules (Section \ref{glp_reps}), and use these to describe the irreducible $\GL(nP)$-modules in Section \ref{glnp_reps}. From that discussion, we will obtain the following theorem.

\begin{theorem}
    The irreducible representations of $\GL(m + nP)$ correspond to pairs of representations of $\GL(m)$ and $\GL(nP)$; in other words, integers $\lambda_1 \ge \cdots \ge \lambda_m$ and $L_{\mu_1}, \dots, L_{\mu_n}$ where $L_{\mu_i}$ is an irreducible representation of $\GL(P)$ of degree $\mu_i$ and $\mu_1 \ge \cdots \ge \mu_n$.
\end{theorem}

\begin{question}
    What are the blocks of $\Rep(\GL(m + nP))$?
\end{question}

\subsubsection{Representations of $\GL(P)$}\label{glp_reps}

We first find all finite-dimensional irreducible $\GL(P)$-modules. 

Recall from Section \ref{examples} that as a functor, $\GL(P)$ takes a commutative algebra $A$ to the group of matrices of the form
\begin{equation*}
    \begin{pmatrix}
        a & a' \\ b & a + b'
    \end{pmatrix},
    a \in A^\times, b \in A.
\end{equation*}

We will denote such an element of $\GL(P)(A)$ by $(a, b)$. 

\begin{proposition}\label{glp_symnp_matrix}
    The $\GL(P)$-action on $\Sym^n (P)$ (induced by the action of $\GL(P)$ on $P$) is given by
    \begin{equation*}
    (a, b) \mapsto \begin{cases}
       \begin{pmatrix}
         a^n & a^{n - 1} a'\\
         a^{n - 1} b + \frac{n - 1}{2} a^{n - 2} a' (b' + a) & a^n + a^{n - 1}
         b'
       \end{pmatrix} & n \text{ odd} \\
       \begin{pmatrix}
         a^n & a^{n - 1} a'\\
         \frac{n}{2} a^{n - 2} a' (b' + a) & a^n + a^{n - 1} b' + a^{n - 2}
         ba'
       \end{pmatrix} & n \text{ even} 
     \end{cases}.
     \end{equation*}
\end{proposition}

\begin{proof}
Throughout this proof, we use the identification $\Sym^n(P) \otimes A = \Sym^n_A(P \otimes A)$ to explicitly describe the $\GL(P)$-action on $\Sym^n(P)$.

    Say that $P$ has basis $x, x'$. Then $\Sym^n(P)$ has basis $x^n, x^{n - 1}x'$. If $n$ is odd, it suffices to define the action of $(a,b)$ on $x^n$; if $n$ is even, we have to compute the action on both. 

    On $P$, $(a, b)$ takes $x \otimes 1 \mapsto x \otimes a + x' \otimes b$. So
    \begin{equation*}
        x^n \otimes 1 \mapsto (x \otimes a + x' \otimes b)^n = (x \otimes a)^n + \sum_{k = 0}^{n - 1} (x \otimes a)^k \otimes (x' \otimes b) \otimes (x \otimes a)^{n - k - 1}.
    \end{equation*}
    We have
    \begin{equation*}
        (x \otimes a)^k \otimes (x' \otimes b) \otimes (x \otimes a)^{n - k - 1} = x^{n - 1}x' \otimes (ba^{n-1} + kba^{n-2}a'b'),
    \end{equation*}
    so
    \begin{align*}
        \sum_{k= 0}^{n-1} (x \otimes a)^k \otimes (x' \otimes b) \otimes (x \otimes a)^{n - k - 1}
        &= x^{n - 1}x' \otimes \left(nba^{n-1} + \sum_{k =0}^{n-1}ka^{n-2}a'b' \right) \\
        &= x^{n-1}x' \otimes \left(nba^{n-1} + \frac{n(n-1)}{2} a^{n-2}a'b' \right).
    \end{align*}
    Now we can recursively compute $(x \otimes a)^n$: write
    \begin{equation*}
        (x \otimes a)^n = x^n \otimes a^n + x^{n - 1}x' \otimes f_n(a),
    \end{equation*}
    then
    \begin{equation*}
        (x \otimes a)^{n + 1} = x^{n + 1} \otimes a^{n + 1} + x^n x' \otimes n a^n a' + x^n x' \otimes f_n(a)a,
    \end{equation*}so
    \begin{equation*}
        f_{n + 1}(a) = f_n(a)a + na^n a', f_1(a) = 0 \implies f_n(a) = \frac{n(n-1)}{2} a^{n - 1}a'.
    \end{equation*}
    Combining these, we get
    \begin{equation}\label{symnp}
        (x \otimes a + x' \otimes b)^n = x^n \otimes a^n + x^{n-1}x' \otimes \left( nba^{n-1} + \frac{n(n-1)}{2} a^{n-2}a'(a+b')\right).
    \end{equation}
    If $n$ is odd, $\frac{n(n-1)}{2} = \frac{n-1}{2}$. Also, $(a, b)(x^{n-1}x') = d((a,b)x^n)$, so we get
    \begin{equation*}
        (a, b) \mapsto \begin{pmatrix}
         a^n & a^{n - 1} a'\\
         a^{n - 1} b + \frac{n - 1}{2} a^{n - 2} a' (b' + a) & a^n + a^{n - 1}
         b'
       \end{pmatrix}.
    \end{equation*}

    If $n$ is even, we also need to compute the action of $(a, b)$ on $x^{n-1}x'$. Thus we need to simplify the expression
    \begin{equation*}
        (x \otimes a' + x' \otimes (a + b')) \otimes (x \otimes a + x' \otimes b)^{n - 1}.
    \end{equation*}
    Using \eqref{symnp}, we get that this is
    \begin{equation*}
        x^n \otimes a^{n-1}a' + x^{n-1}x' \otimes a^n(1 + a^{-1}b' + a^{-2}ba').
    \end{equation*}
    Therefore, the $\GL(P)$ action is given by
    \begin{equation*}
        (a, b) \mapsto \begin{pmatrix}
         a^n & a^{n - 1} a'\\
         \frac{n}{2} a^{n - 2} a' (b' + a) & a^n + a^{n - 1} b' + a^{n - 2}
         ba'
       \end{pmatrix}.
    \end{equation*}
\end{proof}

\begin{corollary}
    $\Sym^n P$ is an irreducible representation of $\GL(P)$ when $n \not\equiv 0 \pmod{4}$. If $n \equiv 0 \pmod{4}$, $\Sym^n P$ is an extension of the representation $(a, b) \mapsto a^n$ by
    \begin{equation*}
        (a, b) \mapsto a^n + a^{n-1}b' + a^{n-2}ba'.
    \end{equation*}
\end{corollary}

\begin{definition}
    Let $T_n = \Sym^n P$ for $n \not\equiv 0 \pmod{4}$. Let $\chi$ be the character $(a, b) \mapsto a^4$ and $\xi$ be the character $(a, b) \mapsto 1 + a^{-1}b' + a^{-2}ba'$. We define $T_n$ for $n < 0$ by setting $n$ to be negative in Proposition $\ref{glp_symnp_matrix}$, or equivalently, $T_n := \chi^{-1} \otimes T_{n + 4}$.
\end{definition}
\begin{corollary}
The duals of $\xi, \chi, T_n$ are $\xi^* = \xi$, $\chi^* = \chi^{-1}$, $T_1^* = \xi \otimes T_{-1}$, $T_2^* = T_{-2}$, and $T_3^* = \xi \otimes T_{-3}$.
\end{corollary}
\begin{proof}
    The proof is by direct computation: namely, the transpose of the inverse of $\begin{pmatrix} a & a' \\ b & a + b' \end{pmatrix}$ can be computed by 
    \begin{align*}
        \left(\begin{pmatrix} a & a' \\ b & a + b' \end{pmatrix}^{-1}\right)^T &= \left(\begin{pmatrix}
            a & a' \\ 0 & a
        \end{pmatrix}^{-1} 
        \begin{pmatrix}
            1 & 0 \\ ba^{-1} & 1 + (ba^{-1})'
        \end{pmatrix}^{-1} \right)^T\\
        &= \left(\begin{pmatrix}
            a^{-1} & a^{-2}a' \\ 0 & a^{-1}
        \end{pmatrix}
        \begin{pmatrix}
            1 & 0 \\
            ba^{-1}(1 + (ba^{-1})') & 1 + (ba^{-1})'
        \end{pmatrix}\right)^T \\
        &= \begin{pmatrix}
            1 + (ba^{-1})' & 0 \\
            ba^{-1}(1 + (ba^{-1})') & 1
        \end{pmatrix}
        \begin{pmatrix}
            a^{-1} & a^{-2}a' \\ 0 & a^{-1}
        \end{pmatrix} \\
        &= \xi(a,b) \otimes \begin{pmatrix}
            1 & 0 \\
            ba^{-1} & 1 + (ba^{-1})'
        \end{pmatrix}
        \begin{pmatrix}
            a^{-1} & a^{-2}a' \\ 0 & a^{-1}
        \end{pmatrix} \\
        &= \xi(a, b) \otimes T_{-1}(a,b).
    \end{align*}
\end{proof}

\begin{proposition}
    The distinct irreducible representations of $\GL(P)$ are exactly the $T_n$ for $n \not\equiv 0 \pmod{4}$, $\chi^k$ for $k \in \mathbb{Z}$, $\xi \otimes \chi^k$ for $k \in \mathbb{Z}$, and $\xi \otimes T_n$ for $n \not\equiv 0, 2 \pmod{4}$. We have the following composition series for tensor products, which we can extend to arbitrary relations for $T_i
  \otimes T_j$ using that $T_i = \chi \otimes T_{i - 4}$:
  \begin{align*}
    \xi^2 & = \1 \\
    \chi T_{i - 4} & = T_i \\
    T_1^2 & = [T_2, T_2]\\
    T_1 T_2 & = [\xi T_3, T_3]\\
    T_1 T_3 & = [\xi \chi, \chi, \chi, \xi \chi]\\
    T_2^2 & = [\xi \chi, \chi + \chi, \xi \chi]\\
    T_2 T_3 & = [\chi T_1, \xi \chi T_1]\\
    T_3^2 & = [\chi T_2, \chi T_2] .
  \end{align*}
  Here the head is first and the socle is last. For the remainder of the section, we drop the tensor product sign between representations.
\end{proposition}
\begin{proof}
    The category of representations of $\GL(P)$ is generated by $P$, so it suffices to find all subquotients of $T_1^{ r} T_1^{* s}$. It's easy to see that $\xi^2 = \1$. Since $T_1^* = \xi T_{-1} = \chi^{-1} \xi T_3$, it suffices to look at subquotients of $T_1^{ r} T_3^{ s}$. 
    
    We have an exact sequence
    \begin{equation*}
        0 \to \wedge^2 T_1 \to T_1^{ 2} \to S^2 T_1 \to 0
    \end{equation*}
    and $S^2 T_1 = \wedge^2 T_1$, so $T_1^2 = T_2$.
    Also,
    \begin{equation*}
        0 \to \wedge^3 T_1 \to T_1 \otimes \wedge^2 T_1 \to S^{(2, 1)}T_1 \to 0.
    \end{equation*}
    Note that $\wedge^3 T_1 = S^3 T_1 = T_3$, while we can explicitly write $S^{(2, 1)} T_1$ as a quotient of $T_1 \otimes \wedge^2 T_1$ by the exchange relations. This quotient is the representation
    \begin{equation*}
        (a, b) \mapsto \begin{pmatrix}
       a^3 + a^2 a' + a^2 b' + aba' & a^2 a' + aa' b'\\
       ba^2 + bab' + b^2 a' & a^3 + aba' + ba' b'
     \end{pmatrix}
    \end{equation*}
    which is $\xi T_3$.
    
    The rest we compute explicitly. The representation of $\GL (P)$ on
  $T_1 T_3$ (with basis $e_1, e_2, e_3, e_4$, $d e_1 = e_2 + e_3$, $d
  e_2 = d e_3 = e_4$, and $d e_4 = 0$) is
  \begin{equation*} (a, b) \mapsto \begin{pmatrix}
       a^4 & a^3 a' & a^3 a' & 0\\
       a^3 b + a^2 a' b' & a^4 + a^3 b' & a^2 ba' & a^3 a' + a^2 a' b'\\
       a^3 b + a^2 a' b' & a^2 ba' & a^4 + a^3 b' & a^3 a' + a^2 a' b'\\
       a^3 b' + a^2 b^2 + a^2 ba' + aba' b' & a^3 b + a^2 bb' & a^3 b + a^2
       bb' & a^4
     \end{pmatrix} . \end{equation*}
  Changing basis to $e_2 + e_3, e_4, e_1, e_2$, we get the upper triangular matrix
  \begin{equation*} \begin{pmatrix}
       a^4 + a^3 b' + a^2 ba' & a^3 a' + a^2 a' b' & a^3 b + a^2 a' b' & a^2
       ba'\\
       0 & a^4 & a^2 b^2 + a^2 ba' + a^3 b' + aba' b' & a^3 b + a^2 bb'\\
       0 & 0 & a^4 & a^3 a'\\
       0 & 0 & 0 & a^4 + a^3 b' + a^2 ba'
     \end{pmatrix} \end{equation*}
  implying that $T_1 T_3 = 2 \xi \chi + 2 \chi$ in the Grothendieck ring, as
  desired.
  
  The representation of $\GL (P)$ on $T_2 \otimes T_2$ (with basis $e_1,
  e_2, e_3, e_4$, $d e_i = 0$) is
  \begin{equation*} (a, b) \mapsto \begin{pmatrix}
       a^4 & a^3 a' & a^3 a' & 0\\
       a^2 a' b' + a^3 a' & a^4 + a^3 b' + a^2 ba' & 0 & a^3 a' + a^2 a' b'\\
       a^2 a' b' + a^3 a' & 0 & a^4 + a^3 b' + a^2 ba' & a^3 a' + a^2 a' b'\\
       0 & a^3 a' & a^3 a' & a^4
     \end{pmatrix} . \end{equation*}
  Changing basis to $e_2 + e_3, e_4, e_1, e_2$, we get the upper triangular
  matrix
  \begin{equation*} \begin{pmatrix}
       a^4 + a^3 b' + a^2 ba' & a^3 a' + a^2 a' b' & a^3 b + a^2 a' b' & 0\\
       0 & a^4 & 0 & a^3 a'\\
       0 & 0 & a^4 & a^3 a'\\
       0 & 0 & 0 & a^4 + a^3 b' + a^2 ba'
     \end{pmatrix} \end{equation*}
  so $T_2^2 = 2 \xi \chi + 2 \chi$ in the Grothendieck ring.
  
  The representation of $\GL (P)$ on $T_2 T_3$ with basis $e_1,
  e_2, e_3, e_4, d e_1 = e_2, d e_3 = e_4$ is
  \begin{equation*} (a, b) \mapsto \begin{pmatrix}
       a^5 & a^4 a' & a^4 a' & 0\\
       a^4 b + a^3 a' b' + a^4 a' & a^5 + a^4 b' & a^3 ba' & a^4 a' + a^3 a'
       b'\\
       a^4 a' + a^3 a' b' & 0 & a^5 + a^4 b' + a^3 ba' & a^4 a' + a^3 a' b'\\
       a^3 ba' + a^2 ba' b' & a^4 a' & a^4 b + a^4 a' + a^3 bb' + a^2 b^2 a' &
       a^5 + a^3 ba' + a^2 ba' b'
     \end{pmatrix} \end{equation*}
  and, changing basis to $e_1, e_2, e_2 + e_3, e_4$, we get the block lower
  triangular matrix
  \begin{equation*} \begin{pmatrix}
       a^5 & a^4 a' & 0 & 0\\
       a^4 b & a^5 + a^4 b' & 0 & 0\\
       a^4 a' + a^3 a' b' & 0 & a^5 + a^4 b' + a^3 ba' & a^4 a' + a^3 a' b'\\
       a^3 ba' + a^2 ba' b' & a^4 a' & a^4 b + a^3 bb' + a^2 b^2 a' & a^5 +
       a^3 ba' + a^2 ba' b'
     \end{pmatrix} . \end{equation*}
  The upper left block is $T_5 = \chi T_1$, while the lower right block is
  $\xi \chi T_1$.
  
  Finally, for $T_3 \otimes T_3$,
  \begin{equation*} (a, b) \mapsto \begin{pmatrix}
       a^6 & a^5 a' & a^5 a' & 0\\
       a^5 b + a^4 a' b' & a^6 + a^5 b' & a^4 ba' & a^5 a' + a^4 a' b'\\
       a^5 b + a^5 a'  & a^4 ba' & a^6 + a^5 b' & a^5 a' + a^4 a' b'\\
       a^5 b' + a^4 b^2 & a^5 b + a^4 bb' + a^5 a' & a^5 b + a^4 bb' & a^6
     \end{pmatrix} \end{equation*}
  and changing basis to $e_1, e_2, e_2 + e_3, e_4$, we get
  \begin{equation*} \begin{pmatrix}
       a^6 & a^5 a' & 0 & 0\\
       a^5 a' + a^4 a' b' & a^6 + a^5 b' + a^4 ba' & 0 & 0\\
       a^5 b + a^5 a'  & a^4 ba' & a^6 + a^5 b' + a^4 ba' & a^5 a' + a^4 a'
       b'\\
       a^5 b' + a^4 b^2 & a^5 b + a^5 a' + a^4 bb' & a^5 a' & a^6
     \end{pmatrix} \end{equation*}
  so $T_3^2 = 2 T_6 = 2 \chi T_2$ in the Grothendieck ring.

  Hence, the finite-dimensional irreducible representations of $\GL(P)$ are exactly $\chi^n$, $\xi \otimes \chi^n$, $T_n$ for $n \not\equiv 0 \pmod{4}$, and $\xi \otimes T_n$ for $n \equiv 1, 3 \pmod{4}$.
\end{proof}

In particular, we can define a notion of degree on the irreducible $\GL(P)$-representations via the inclusion of $\mbb{G}_m \subset \GL(P)$, $a \mapsto \begin{pmatrix} a & 0 \\ 0 & a \end{pmatrix}$ for $a \in A^\times$: the degree of an irreducible $\GL(P)$-representation is its degree as a $\mbb{G}_m$-representation. Hence, we can characterize $\Rep(\GL(P))$ as follows:

\begin{proposition}
    The category $\Rep(\GL(P))$ is $\mbb{Z}$-graded by degree, and it has exactly one block in each degree: 
    \begin{equation*}
         B_n = \begin{cases}
       \langle \chi^{n / 4}, \xi \chi^{n / 4} \rangle & n \equiv 0
       \pmod{4}\\
       \langle T_n, \xi T_n \rangle & n \equiv 1, 3 \pmod{4}\\
       \langle T_n \rangle & n \equiv 2 \pmod{4}
     \end{cases}. 
    \end{equation*}
\end{proposition}

\begin{proof}
  These are exactly the simple objects of each degree, so it suffices to show
  that when $n \not\equiv 2 \pmod{4}$, there are nontrivial extensions
  between the generators listed.
  
  For degree $4 k$, note that we have a non-split exact sequence $0 \rightarrow \chi^k
  \rightarrow T_{4 k} \rightarrow \xi \chi^{k} \rightarrow 0$. For
  degree $4 k + 1$, note that $T_4$ is a nontrivial extension of $\chi$ by $\xi$, so $\chi^{k - 1} T_4 T_1$ is a nontrivial extension
  of $\chi^k T_1$ by $\xi \chi^k T_1$. For
  degree $4 k + 3$, we showed above that $T_1 T_2$ is a nontrivial extension of $T_3$ by $\xi T_3$, so $\chi^k T_1 T_2$ is a nontrivial extension of
  $\chi^k T_3$ by $\xi \chi^k \otimes T_3$.
\end{proof}

We can now describe the Ext quiver for the degree 0 block. Let $G = \GL(P)$; then we have $\mcal{O}(G) = k[A, B, A^{-1}, A', B']/(A')^2 = (B')^2 = 0$, where $A$ corresponds to the function $(a, b) \mapsto a$ and $B$ to $(a, b) \mapsto b$. Then the characters of $G$ can be treated as elements of $\mcal{O}(G)$, e.g. $\1$ as $1 \in \mcal{O}(G)$ and $\xi$ as $1 + (BA^{-1})' \in \mcal{O}(G)$.
\begin{proposition}
Let $\xi \in \mathcal{O}(G)$ be $1 + (BA^{-1})'$. Elements of $\Ext(\1, \1)$ correspond to primitive elements in $\mathcal{O}(G)$, while elements of $\Ext(\1, \xi)$ correspond to skew-primitive elements $f \in \mathcal{O}(G)$ with $\Delta(f) = 1 \otimes f + f \otimes \xi$.
\end{proposition}
\begin{proof}
    See Proposition 1.9.12, \cite{etingof_tensor_2015}.
\end{proof}

Since $\Ext(\1, \1) \cong \Ext(\xi, \xi)$ and $\Ext(\1, \xi) \cong \Ext(\xi, \1)$ via tensoring an extension by $\xi$, it suffices to compute primitive elements and skew-primitive elements $f \in \mcal{O}(G)$ such that $\Delta(f) = 1 \otimes f + f \otimes \xi$.

\begin{proposition}
    $\Ext(\1, \1)$ and $\Ext(\xi, \xi)$ are infinite-dimensional, while $\Ext(\1, \xi)$ and $\Ext(\xi, \1)$ are 1-dimensional.
\end{proposition}
\begin{proof}
Either from direct computation or looking at the $\GL(P)$-action on $\Sym^4(P)$, we have that $A^{-1}A'$ corresponds to an element in $\Ext(\1, \xi)$. Then $A^{-1}A'\xi = A^{-1}A' + A^{-2}B'A'$ is in $\Ext(\xi, \1)$.

For elements in $\Ext(\1, \1)$, we note that the fourth power map is a homomorphism, so $\Delta(A^4) = A^4 \otimes A^4$ and $\Delta(B) = B^4 \otimes A^4 + A^4 \otimes B^4$. Hence, any element of the form $(BA^{-1})^{4 \cdot 2^k}$ is primitive.

We claim that these are the only (skew)-primitive elements up to coboundaries. 

Since as schemes $\GL(P) \cong M_1 \times H_1$, $H_1 := \mathbb G'_m$ via
\begin{equation*}
    \begin{pmatrix}
    a & a' \\ b & a + b' \end{pmatrix} = 
    \begin{pmatrix}
        1 & 0 \\ ba^{-1} & 1 + (ba^{-1})'
    \end{pmatrix}
    \begin{pmatrix}
        a & a' \\ 0 & a
    \end{pmatrix},
\end{equation*}
a representation of $\GL(P)$ is determined by its restriction to $M_1$ and $H_1$. An extension of $\1$ by $\1$ must still be one when restricted to $M_1$ and $H_1$. An extension of $\1$ by $\xi$ restricts to an extension of trivial representations for $H_1$, and an extension of $\1$ by $x \mapsto 1 + x'$ for $M_1$. So it suffices to classify primitive elements in $\mathcal{O}(H_1)$ and both primitive and skew-primitive elements in $\mathcal{O}(M_1)$.

Recall that $\mathcal{O}(H_1) = k[X, X^{-1}, X']$ with $\Delta(X) = X \otimes X$. It is easy to see that the only primitive elements in this Hopf algebra are those in the span of $X^{-1}X'$.

Recall that $\mathcal{O}(M_1) = k[Y, Y']$ with $\Delta(Y) = Y \otimes 1 + (1 + Y') \otimes Y$. Then we can compute $\Delta(Y^n)$ and $\Delta(Y^n Y')$ by using that $Y^{4 \cdot 2^k}$ is primitive and computing $\Delta(Y^n)$ for $0 \le n < 3$, e.g. $\Delta(Y^2) = 1 \otimes Y^2 + Y^2 \otimes 1 + Y' \otimes Y'$. Note the $Y' \otimes Y'$ term appears because the braiding is nontrivial, so
\begin{equation*}
    ((1 + Y') \otimes Y)(Y \otimes 1) = (Y + YY') \otimes Y + Y' \otimes Y'.
\end{equation*}
It's easy to see that nothing in the span of $Y^n, Y^n Y'$ is primitive except $\text{span}(Y^{4 \cdot 2^k})$, and none are skew-primitive except $\text{span}(Y, Y')$. 

Let $V$ be a $\GL(P)$-representation with $0 \to \1 \to V \to \1 \to 0$.
Restricting $V$ to $H_1$, we obtain
\begin{equation*}
    \begin{pmatrix}a & a' \\ 0 & a \end{pmatrix} \mapsto \begin{pmatrix} 1 & \lambda a^{-1}a' \\ 0 & 1 \end{pmatrix}.
\end{equation*}
Restricting $V$ to $M_1$, we obtain
\begin{equation*}
    \begin{pmatrix} 1 & 0 \\ ba^{-1} & 1 + (ba^{-1})' \end{pmatrix} \mapsto \begin{pmatrix} 1 & \sum \lambda_k (ba^{-1})^{4 \cdot 2^k} \\ 0 & 1 \end{pmatrix}.
\end{equation*}
By direct computation, we can now check that 
\begin{equation*}
    \begin{pmatrix} a & a' \\ b & a + b'\end{pmatrix} \mapsto \begin{pmatrix} 1 & \lambda a^{-1}a' +\sum \lambda_k (ba^{-1})^{4 \cdot 2^k}  \\ 0 & 1 \end{pmatrix}
\end{equation*}
is not a $\GL(P)$-representation unless $\lambda = 0$. Hence, the only extensions of $\1$ by itself are the trivial extension and those corresponding to the span of $(BA^{-1})^{4 \cdot 2^k}$ in $\mathcal{O}(\GL(P))$.

Likewise, an extension of $\1$ by $\xi$ has
\begin{equation*}
    \begin{pmatrix}a & a' \\ 0 & a \end{pmatrix} \mapsto \begin{pmatrix} 1 & \lambda a^{-1}a' \\ 0 & 1 \end{pmatrix}
\end{equation*}
while
\begin{equation*}
    \begin{pmatrix} 1 & 0 \\ ba^{-1} & 1 + (ba^{-1})' \end{pmatrix} \mapsto \begin{pmatrix} 1 & \lambda_1 (ba^{-1})' + \lambda_2 ba^{-1} \\ 0 & 1 + (ba^{-1})' \end{pmatrix}.
\end{equation*}
Again, we can check that we only produce a $\GL(P)$ representation when $H_1$ acts by $\begin{pmatrix} 1 & \lambda a^{-1}a' \\ 0 & 1 \end{pmatrix}$ and $\lambda_2 = 0$. However, we can check that $(ba^{-1})' = 1 + \xi$ is a coboundary. So $\Ext(\1, \xi)$ is 1-dimensional and spanned by $A^{-1}A'$.
\end{proof}



We can also see how these irreducible $\GL(P)$ representations restrict to representations of $M_1$ and $H_1$.
\begin{proposition}
    If $n$ is odd, $T_n, \xi T_n$ both restrict to 
$L_n$, the irreducible representation of $H_1$ labeled by $n$. As an $M_1$ representation, $T_n$ restricts to an extension of $L_\xi$ (the nontrivial irreducible representation of $M_1$) by $\1$, while $\xi T_n$ restricts to an extension of $\1$ by $L_\xi$. If $n \equiv 2 \pmod{4}$, $T_n$ restricts to an extension of $L_n$ by itself as a representation of $H_1$, and to a direct sum of $\1$ and $M_\xi$ as a representation of $M_1$. If $n \equiv 0 \pmod{4}$, $\chi^n, \xi \chi^n$ restrict to $L_n$ as representations of $H_1$. As a representation of $M_1$, $\chi^n$ restricts to $\1$, while $\xi \chi^n$ restricts to $M_\xi$.
\end{proposition}
\begin{proof}
Via direct computation – restricting to $H_1$ sets $b = 0$, and restricting to $M_1$ sets $a = 1$.
\end{proof}

\begin{remark}
    We can also consider the induced modules $\Hom_{H_1}(V, \mcal{O}(G))$ for $V$ an $H_1$-representation and $\Hom_{M_1}(W, \mcal{O}(G))$ for $W$ an $M_1$-representation. The isomorphism of schemes $G \cong M_1 \times H_1$ implies $\mcal{O}(G) \cong \mcal{O}(M_1) \otimes \mcal{O}(H_1)$, so
\begin{equation*}
    \Hom_{H_1}(V, \mcal{O}(G)) = V^* \otimes \mcal{O}(M_1), \Hom_{M_1}(W, \mcal{O}(G)) = W^* \otimes \mcal{O}(H_1).
\end{equation*}
Recall that $L_\xi$ is the irreducible $M_1$-representation defined by $a \mapsto 1 + a'$. The invariant functions we get in $\Hom_{M_1}(L_\xi, \mcal{O}(G))$ generate the $G$-representation $\xi$; indeed, $\xi(a, b) = 1 + (ba^{-1})'$. Recall that $L_n$ is the irreducible $H_1$-representation corresponding to the integer $n$. We can check that the $H_1$-invariant functions $L_n \to \mcal{O}(G)$ corresponding to $L_n^* \otimes 1$ generate either $T_n$ or $\chi^{n/4}$ depending on whether $n \equiv 0 \pmod{4}$ or not.
\end{remark}

\subsubsection{Explicit structure of representations of $\GL(\1+P)$ and their tensor products}

By Proposition \ref{glm+np_irrep}, we know that finite-dimensional irreducible representations of $\GL(\1 + P)$ correspond to pairs $(n, T)$ where $n \in \mathbb{Z}$ and $T$ is an irreducible $\GL(P)$-module. In this section, we give the weight decomposition for all these irreducibles and use this decomposition to compute tensor products of the irreducible $\GL(\1 + P)$-modules.

Recall from Corollary \ref{glm+np_lu} that there is a subgroup $B_{1, 1} \in \GL(\1 + P)$ of block upper triangular matrices of the form
\begin{equation*}
    \begin{pmatrix}
        \1 \otimes \1^* & \1 \otimes P^* \\
        0 & P \otimes P^*
    \end{pmatrix}.
\end{equation*}
The irreducible representations of $B_{1, 1}$ are also in bijection with pairs $(n, T)$ where $n \in \mathbb{Z}$, $T$ an irreducible $\GL(P)$-module. Thus given an irreducible $B_{1, 1}$-representation $V$ with highest weight $(n, T)$, we can construct the corresponding $\GL(\1 + P)$ representation with highest weight $(n, T)$ by finding a $B_{1, 1}$-invariant function $f: V \to \mcal{O}(\GL(\1 + P))$, and then looking at the $\GL(\1+P)$-representation generated by $f$. To find such an invariant function, we identify $\Hom_{B_{1, 1}}(V, \mcal{O}(\GL(\1 + P)))$ with $V^{*} \otimes \Sym (P)$. We then analyze the $\GL(\1 + P)$-representation using Sage to determine its $\GL(\1) \times \GL(P)$ weight decomposition.

Let $L (n, T)$ be the irreducible representation of $\GL (1 + P)$ corresponding to $n$ an
integer (i.e. a representation of $\GL_1$) and $T$ an irreducible representation of
$\GL (P)$. Because $L (2, \1)$, $L(0, \chi)$ are one-dimensional, it suffices to describe $L(n, T)$ when $n = 0, 1$ and $0 \le \deg T < 4$. 

Write elements of $\GL (\1 + P) (A)$ in the form
\begin{equation*} \begin{pmatrix}
     a & c & c'\\
     b' & e & e'\\
     b & f & e + f'
   \end{pmatrix} \end{equation*}
where the first row/column corresponds to $\1$ and the second and third
rows/columns correspond to $P$. In particular $a' = 0$ and $a, e$ are
invertible.

Let $A, B, C, E, F$ be the matrix functions taking an element to $a, b, c, e,
f$ respectively, so $\mcal{O}(G) = k [A, B, C, E, F, A^{- 1}, E^{- 1}] / (d A = 0)$.

\begin{proposition}
Explicit invariants for some $B_{1, 1}$-representations are as follows:
\begin{center}
\bgroup \small
\begin{tabular}{|c|c|}
\hline
    Highest Weight & Invariant \\
    \hline
    $(n, \1)$ & $A^n$ \\
    \hline
    $(0, \chi^n)$ & $E^{4n}$ \\
    \hline
    $(0, T_i), 1 \le i \le 3$ & $A^{-1} E^i B'$ \\
    \hline
    $(0, \xi \chi)$ & \makecell{$E^4 + E^3F' + E^2FE' + $ \\ $A^{-1}(CE^3B' + BE^3C' + E^2FB'C' +$ \\ $ BCE^2E' + CE^2B'F' + EB'C'D'F')$} \\
    \hline
    $(1, \xi\chi)$ & \makecell{$A(E^4 + E^3F' + E^2FE') + $ \\ $CE^3B' + BE^3C' + E^2FB'C' +$ \\ $ BCE^2E' + CE^2B'F' + EB'C'D'F'$} \\
    \hline
    $(1, T_i), 1 \le i \le 3$ & $E^i B'$ \\
    \hline
\end{tabular}
\egroup
\end{center}
\end{proposition}

We can then check that $L(1, \xi \chi)$ is one-dimensional. Therefore, computing the $\GL(\1+P)$-representations with highest weights in the above list is sufficient; all other representations are tensor products of those with $L(2n, \1)$, $L(0, \chi^n)$, and $L(1, \chi \xi)$.

The weight decompositions for $\GL(\1 + P)$ representations with highest weight $(n, T)$, $0 < \deg T \le 4$, $0 \le n < 2$, are listed below. Extensions are denoted by $\times$.
\begin{itemize}
  \item $L (2, \1)$, $L (0,\chi)$, and $L (1, \xi \chi)$ are one-dimensional.
  
  \item $L (1, \1)$ is 3-dimensional: \begin{tabular}{|c|}
    \hline
    $(1, \1)$\\
    \hline
    $(0, T_1)$\\
    \hline
  \end{tabular}
  
  \item $L (0, T_1)$ is $8$-dimensional: \begin{tabular}{|c|}
    \hline
    (0,$T_1$)\\
    \hline
    $(- 1, T_2) \times (- 1, T_2)$\\
    \hline
    $(- 2, \xi T_3)$\\
    \hline
  \end{tabular}
  
  \item $L (0, T_2)$ is 4-dimensional: \begin{tabular}{|c|}
    \hline
    $(0, T_2)$\\
    \hline
    $(- 1, \xi T_3)$\\
    \hline
  \end{tabular}
  
  \item $L (0, T_3)$ is 8-dimensional: \begin{tabular}{|c|}
    \hline
    $(0, T_3)$\\
    \hline
    $(- 1, \xi \chi) \times (- 1, \chi) \times (- 1, \chi) \times (- 1,
    \xi \chi)$\\
    \hline
    $(- 2, \xi \chi T_1)$\\
    \hline
  \end{tabular}
  
  \item $L (1, T_1)$, $L (1, T_2)$ are 4-dimensional: \begin{tabular}{|c|}
    \hline
    $(1, T_i)$\\
    \hline
    $(0, T_{i + 1})$\\
    \hline
  \end{tabular}
  
  \item $L (1, T_3)$ is 3-dimensional: \begin{tabular}{|c|}
    \hline
    $(1, T_3)$\\
    \hline
    $(0, \xi \chi)$\\
    \hline
  \end{tabular}
  
  \item $L (0, \xi \chi) \cong L(1, T_3)^* \otimes L(\chi)^{\otimes 2}$ is 3-dimensional: \begin{tabular}{|c|}
    \hline
    $(0, \xi \chi)$\\
    \hline
    $(- 1, \xi \chi T_1)$\\
    \hline
  \end{tabular}
  
  \item $L (0, \xi \chi^{- 1} T_3) \cong L(1, \1)^*$ is 3-dimensional: \begin{tabular}{|c|}
    \hline
    $(0, \xi \chi^{- 1} T_3)$\\
    \hline
    $(- 1, \1)$\\
    \hline
  \end{tabular}
  
  \item $L (0, \xi \chi^{- 1} T_1) \cong L(1, T_2)^*$ is 4 dimensional: \begin{tabular}{|c|}
    \hline
    $(0, \xi \chi^{- 1} T_1)$\\
    \hline
    $(- 1, \chi^{- 1} T_2)$\\
    \hline
  \end{tabular}
  
  \item $L (1, \xi T_1) \cong L (1 \comma \xi \chi) \otimes L (\chi^{- 1})
  \otimes L (0, T_1)$ is 8-dimensional: \begin{tabular}{|c|}
    \hline
    (1,$\xi T_1$)\\
    \hline
    $(0, T_2) \times (0, T_2)$\\
    \hline
    $(- 1, T_3)$\\
    \hline
  \end{tabular}
  
  \item $L (1, \xi T_3) \cong L (1 \comma \xi \chi) \otimes L (\chi^{-
  1}) \otimes L (0, T_3)$ is 8-dimensional: \begin{tabular}{|c|}
    \hline
    $(1, \xi T_3)$\\
    \hline
    $(0, \chi) \times (0, \xi \chi) \times (0, \xi \chi) \times (0, \xi)$\\
    \hline
    $(- 1, T_1)$\\
    \hline
  \end{tabular}
\end{itemize}

We also compute weight decompositions for tensor products of the above irreducible representations in the table below. These were calculated using
Sage.

\begin{table}
\centering
\bgroup \small
\begin{tabular}{|c|c|c|c|c|c|c|c|}
\hline
   & $(0, T_1)$ & $(0, T_2)$ & $(0, T_3)$ & $(1, 1)$ & $(1, T_1)$ &
  $(1, T_2)$ & $(1, T_3)$\\
  \hline
  $(0, T_1)$ & \makecell{$2 (0, T_2) +$
  \\
  $2 (- 1, \xi T_3) +$
  \\
  $4 (- 1, T_3) +$
  \\
  $4 (- 2, \xi \chi) +$
  \\
  $8 (- 2, \chi) +$
  \\
  $2 (- 3, \chi T_1)$} & \  & \  & \  & \  & \  & \ \\
  \hline
  $(0, T_2)$ & \makecell{$(0, \xi T_3) +$
  \\
  $(0, T_3) +$
  \\
  $4 (- 1, \xi \chi) +$
  \\
  $3 (- 1, \chi) +$
  \\
  $2 (- 2, \xi \chi T_1)$} & \makecell{$2 (0, \xi \chi) +$
  \\
  $2 (0 \comma \chi) +$
  \\
  $2 (- 1, \chi T_1)$} & \  & \  & \  & \  & \ \\
  \hline
  $(0, T_3)$ & \makecell{$2 (0, \xi \chi) +$
  \\
  $2 (0, \chi) +$
  \\
  $2 (- 1, \xi \chi T_1) +$
  \\
  $4 (- 1 \comma \chi T_1) +$
  \\
  $4 (- 2, \chi T_2) +$
  \\
  $2 (- 3, \chi T_3) +$
  \\
  $2 (- 4, \chi^2)$} & \makecell{$(0, \xi \chi T_1) +$
  \\
  $(0, \chi T_1) +$
  \\
  $3 (- 1, \chi T_2) +$
  \\
  $2 (- 2, \xi \chi T_3) +$
  \\
  $2 (- 3, \xi \chi^2)$} & \makecell{$2 (0, \chi T_2) +$
  \\
  $2 (- 1, \xi \chi T_3) +$
  \\
  $4 (- 1, \chi T_3) +$
  \\
  $4 (- 2, \xi \chi^2) +$
  \\
  $8 (- 2, \chi^2) +$
  \\
  $2 (- 3, \chi^2 T_1)$} & \  & \  & \  & \ \\
  \hline
  $(1, 1)$ & \makecell{$(1, T_1) +$
  \\
  $3 (0, T_2) +$
  \\
  $2 (- 1, T_3) +$
  \\
  $2 (- 2, \chi)$} & \makecell{$(1, T_2) +$
  \\
  $2 (0, \xi T_3) +$
  \\
  $2 (- 1, \xi \chi)$} & \makecell{$(1, T_3) +$
  \\
  $3 (0, \xi \chi) +$
  \\
  $4 (0, \chi) +$
  \\
  $2 (- 1, \chi T_1)$} & \makecell{$(2, 1) +$
  \\
  $2 (1, T_1)$} & \  & \  & \ \\
  \hline
  $(1, T_1)$ & \makecell{$2 (1, T_2) +$
  \\
  $3 (0, \xi T_3) +$
  \\
  $(0, T_3) +$
  \\
  $4 (- 1, \xi \chi) +$
  \\
  $(- 1, \chi)$} & \makecell{$(1, \xi T_3) +$
  \\
  $(1, T_3) +$
  \\
  $(0, \xi \chi) +$
  \\
  $2 (0, \chi)$} & \makecell{$2 (1, \xi \chi) +$
  \\
  $2 (1, \chi) +$
  \\
  $3 (0, \xi \chi T_1) +$
  \\
  $(0, \chi T_1) +$
  \\
  $(- 1, \chi T_2)$} & \makecell{$(2, T_1) +$
  \\
  $(1, T_2)$} & \makecell{$2 (2, T_2) +$
  \\
  $2 (1, T_3) +$
  \\
  $2 (0, \chi)$} & \  & \ \\
  \hline
  $(1, T_2)$ & \makecell{$(0, \xi T_3) +$
  \\
  $(0, T_3) +$
  \\
  $3 (- 1, \xi \chi) +$
  \\
  $4 (- 1, \chi) +$
  \\
  $2 (- 2, \chi T_1)$} & \makecell{$2 (1, \xi \chi) +$
  \\
  $2 (1, \chi) +$
  \\
  $2 (0, \xi \chi T_1)$} & \makecell{$(1, \xi \chi T_1) +$
  \\
  $(1, \chi T_1) +$
  \\
  $3 (0, \chi T_2) +$
  \\
  $2 (- 1, \chi T_3) +$
  \\
  $2 (- 2, \chi^2)$} & \makecell{$(2, T_2) +$
  \\
  $2 (1, T_3) +$
  \\
  $2 (0, \chi)$} & \makecell{$(2, \xi T_3) +$
  \\
  $(2, T_3) +$
  \\
  $2 (1, \xi \chi) +$
  \\
  $(1, \chi)$} & \makecell{$2 (2, \xi \chi) +$
  \\
  $2 (2, \chi) +$
  \\
  $2 (1, \chi T_1)$} & \ \\
  \hline
  $(1, T_3)$ & \makecell{$2 (0, \xi \chi) +$
  \\
  $2 (0, \chi) +$
  \\
  $3 (- 1, \xi \chi T_1) +$
  \\
  $(- 2, T_2)$} & \makecell{$(1, \xi \chi T_1) +$
  \\
  $(1, \chi T_1)$} & \makecell{$2 (1, \chi T_2) +$
  \\
  $3 (0, \xi \chi T_3) +$
  \\
  $4 (- 1, \xi \chi^2) +$
  \\
  $(- 1, \chi^2)$} & \makecell{$(2, T_3) +$
  \\
  $(1, \xi \chi)$} & \makecell{$2 (2, \xi \chi) +$
  \\
  $2 (2, \chi) +$
  \\
  $(1, \chi T_1)$} & \makecell{$(2, \xi \chi T_1) +$
  \\
  $(2, \chi T_1)$} & \makecell{$2 (2, \chi T_2) +$
  \\
  $(0, \chi^2)$}\\
  \hline
\end{tabular}
\caption{Weight decompositions for tensor products of irreducible representation of $\GL(1+P)$}
\egroup
\end{table}

\subsubsection{Representations of $\GL(nP)$} \label{glnp_reps}

In this section, we describe the irreducible representations of $\GL(nP)$. 

If we write elements of $\GL(nP)(A)$ as $n \times n$ block matrices
\begin{equation*}
    \begin{pmatrix}
        X_{11} & \cdots & X_{1n} \\
        \vdots & \ddots & \vdots \\
        X_{n1} & \cdots & X_{nn}
    \end{pmatrix}
\end{equation*}
where $X_{ij} \in \End_A(P \otimes A)$, we see that $\GL(nP)$ has a subgroup $B$ consisting of the upper triangular block matrices of this form with the $X_{ii}$ invertible. The irreducible representations of this subgroup $B$ are $n$-tuples of irreducible $\GL(P)$-modules, and per the discussion in Section \ref{highest_weights}, every irreducible representation of $\GL(nP)$ has a unique $B$-highest weight, which is an irreducible $B$-module $V$. The irreducible representations of $\GL(nP)$ correspond to the irreducible $B$-modules $V$ with $\Hom_{B}(V, \mcal{O}(\GL(nP))) \ne 0$.

\begin{lemma}
    If the $n$-tuple of $\GL(P)$-irreps $(L_1, \dots, L_n)$ does not satisfy $\deg L_1 \ge \cdots \ge \deg L_n$, then it cannot be the highest weight of an irreducible $\GL(nP)$-module.
\end{lemma}
\begin{proof}
    We have $\GL(n) \subset \GL(nP)$ via sending an invertible matrix $M \in \GL(n)(A)$ to $\begin{pmatrix} M & 0 \\ 0 & M \end{pmatrix}$. A $\GL(nP)$-module with highest weight $(L_1, \dots, L_n)$ restricts to a $\GL(n)$-module (not necessarily irreducible) with highest weight $\deg L_1, \dots, \deg L_n$, which cannot be finite-dimensional unless $\deg L_1 \ge \cdots \ge \deg L_n$.
\end{proof}

We claim that this is also a sufficient condition.

\begin{lemma}
    The $\GL(nP)$-module with highest weight $(L_1, \dots, L_n)$, $\deg L_1 \ge \cdots \ge \deg L_n$, is finite-dimensional.
\end{lemma}
\begin{proof}
Let $G = \GL(nP)$ and let $N$ be the nilradical of $\mathcal{O}(G)$; then $N$ is the ideal generated by $\im d$ and $\mathcal{O}(G)/N = \mcal{O}(G_*)$ for $G_*$ the closed subgroup such that $G_*(A) = \left\{ \begin{pmatrix} C & 0 \\ D & C \end{pmatrix}\right\}$, $C \in \GL(n)(A)$, $D \in \mfrak{gl}_n(A)$, so $dC = dD = 0$. Let $(D, C) \in G_*(A)$ denote the element
\begin{equation*}
    \begin{pmatrix}
        I & 0 \\ D & I 
    \end{pmatrix}
    \begin{pmatrix}
        C  & 0 \\ 0 &  C
    \end{pmatrix}.
\end{equation*}
Then multiplication on $G_*$ is given by 
\begin{equation*}
    (D_1, C_1)(D_2, C_2) = (D_1 + C_1 D_2 C_1^{-1}, C_1 C_2),
\end{equation*}
so $G_* = \GL(n) \ltimes \mfrak{gl}(n)$ where $\GL(n)$ acts by conjugation. Therefore, irreducible representations of $G_*$ correspond exactly to representations of $\GL(n)$ with trivial $\mfrak{gl}_n$-action. Consider the irreducible $G_*$ representation $V$ with highest weight $(\lambda_1, \dots, \lambda_n)$. The representation $\ind_{G_*}^G(V)$ contains the irreducible $G$-representation with highest weight $(L_1, \dots, L_n)$, $\deg L_i = \lambda_i$. We prove below in Lemma \ref{dist_g_fg} that $\Dist(G)$ is a finitely generated module over $\Dist(G_*)$, which implies that $\ind_{G_*}^G(V)$ is finite-dimensional. Hence, the irreducible $G$-representation with highest weight $(L_1, \dots, L_n)$ is finite-dimensional as well.

\end{proof}

\begin{lemma}
    \label{dist_g_fg}
    $\Dist(G)$ is finitely generated as a $\Dist(G_*)$-module.
\end{lemma}
\begin{proof}
Let $\mcal{O}(\GL(nP)) = k[C_{ij}, D_{ij}, C'_{ij}, D'_{ij}]$, where $C_{ij}$ and $D_{ij}$ correspond to matrix entries in an element $\begin{pmatrix} C & C' \\ D & C + D' \end{pmatrix}$ and $C'_{ij} = dC_{ij}$, $D'_{ij}, = dD_{ij}$, $1 \le i, j \le n$. The inclusion $G_* \subset G$ induces a map $\Dist(G_*) \inj \Dist(G)$ whose image is the elements of $\Dist(G)$ vanishing on $N$.

The augmentation ideal $I \subset \mcal{O}(\GL(nP))$ is generated by $C_{ij} - \delta_{ij}$ where $\delta$ is the Kronecker delta, and $D_{ij}, C'_{ij}, D'_{ij}$ for $1 \le i, j \le n$. For ease of notation, let $C - 1 := \{C_{ij} - \delta_{ij}\}$, $D := \{D_{ij}\}$, $C' := \{C'_{ij}\}$, $D' := \{D'_{ij}\}$. 
Then, the residue classes of monomials in $C - 1$, $D$, $C'$, $D'$ of degree at most $m$ form a basis for $\mcal{O}(G)/I^{m+1}$ and for a fixed monomial $p(C - 1, D, C', D')$, there is a unique $\delta_p \in \Dist(G)$ such that

\begin{equation*}
    \delta_p(I^{m+1}) = 0, \delta_p(p) = 1, \delta_p(q) = 0
\end{equation*}
when $q \ne p$ is another monomial in $C - 1, D, C', D'$ of degree at most $m$. The $\delta_p$ over all monomials $p$ of any nonnegative degree form a basis of $\Dist(G)$.
We claim that as a $\Dist(G_*)$-module, $\Dist(G)$ is generated by $\delta_{s(C', D')}$ where $s$ runs over all monomials in $C'_{ij}, D'_{ij}$. Let $V \subset \Dist(G)$ be the $\Dist(G_*)$-module generated by such $\delta_{s(C', D')}$. We induct on the degree of $p$ to show that $\delta_p$ lies in $V$. The base case, $\deg p = 1$, is clear. If $\deg p \ge 1$, write $p = p_1(C - 1, D)p_2(C', D')$. Recall (e.g. \cite{jantzen_algebraic_groups_2003} 7.7) that for $f_1, \dots, f_n \in I$,
\begin{equation*}
    \Delta(f_1 f_2 \cdots f_n) \in \prod_{i = 1}^n (1 \otimes f_i + f_i \otimes 1) + \sum_{r = 1}^n I^r \otimes I^{n + 1 - r}.
\end{equation*}
Therefore,
\begin{equation*}
    \delta_{p_1(C - 1, D)} \delta_{p_2(C', D')} = \delta_{p} + \mu, \mu \in \text{span}(\delta_q)_{\deg q < m}.
\end{equation*}
By induction, $\mu \in V$; $\delta_{p_1(C - 1, D)} \in \Dist(G_*)$; and $\delta_{p_2(C', D')} \in V$ by definition. Hence, $\delta_p \in V$ as well.

Since $C'_{ij}$ and $D'_{ij}$ square to 0, there are only finitely many monomials in $C'_{ij}, D'_{ij}$. Therefore, $\Dist(G)$ is finitely generated as a $\Dist(G_*)$-module.
\end{proof}

\begin{corollary}
    The irreducible $\GL(nP)$-representations are labeled by $n$-tuples of irreducible $\GL(P)$ representations $(L_1, \dots, L_n)$ where $\deg L_1 \ge \cdots \ge \deg L_n$.
\end{corollary}

\subsubsection{An alternative construction of representations of $\GL(nP)$}

In this section, we describe another approach to constructing representations of $\GL(nP)$.

Recall that multiplication is an isomorphism $\GL(nP) \cong M_n \times H_n$ as schemes, where $H_n(A)$ is the group of $n \times n$ invertible matrices over $A$ under multiplication and $M_n(A)$ is the group of all $n \times n$ matrices with operation $(X, Y) \mapsto X + Y + X'Y$.
\begin{proposition}
    Representations of $H_n$, like for $\GL(n)$, correspond to $n$-tuples of integers $\lambda_1 \ge \cdots \ge \lambda_n$.
\end{proposition}
\begin{proof}
    Inside $H_n$ sits the group of $n \times n$ upper triangular invertible matrices, which we call $B$; irreducible representations of $B$ correspond to $n$-tuples of irreducible representations of $H_1$, which correspond to the integers. There is at most one representation of each $B$-weight, and the same argument as for $\GL(n)$ implies that $(\lambda_1, \dots, \lambda_n)$ only corresponds to an irreducible $H_n$-representation if the sequence is nonincreasing.

    
    To construct a representation of each weight, note that $\GL(n) \subset H_n$ as the subgroup of invertible matrices in the kernel of $d$. Let $\mcal{O}(H_n) = k[X_{ij}, X'_{ij}, \det^{-1}]$, where $X_{ij}$ is the function taking an invertible $n \times n$ matrix to its $ij$th entry, and $X'_{ij} = dX_{ij}$. We claim that for every irreducible (polynomial) $\GL(n)$-representation $V$, $\Hom_{\GL(n)}(V, \mathcal{O}(H_n))$ is nonzero. If $(a_{ij}) \in \GL(n)$ acts on $V$ via the matrix $p_{i'j'}(a_{ij})$ where $1 \le i', j' \le \dim V$ and $p$ is some polynomial, then the map $v_{j'} \mapsto p_{1, j'}(X'_{ij})$ is a $\GL(n)$-invariant map. Since $V$ is a $\GL(n)$-representation and since the matrix entries in elements of $\GL(n)$ are in the kernel of $d$, the $\GL(n)$ action on $p_{1, j'}(X'_{ij})$ is $p_{1, j'}(\sum_k X'_{ik} \otimes e_{kj}) = \sum_{k'} p_{1, k'}(X'_{ij}) \otimes p_{k', j'}(e_{ij}) $. 
    Then there exists at least one $H_n$-representation with the same weight as $V$, so there is exactly one for each partition.
\end{proof}

\begin{proposition}
    Representations of $M_n$ correspond to $n$-tuples of nonnegative integers $\lambda_1 \ge \cdots \ge \lambda_n$.
\end{proposition}
\begin{proof}
    Let $\End_n$ be the affine group defined by
    \begin{equation*}
        \End_n(A) = \{X \in \Mat_{n \times n}(A), X' = 0\} \text { under addition}.
    \end{equation*}
    Then $\End_n \subset M_n$.
    
    Hence, all irreducible $M_n$-representations are subrepresentations of the induced representation $(\1 \otimes \mcal{O}(M_n))^{\End_n}$; this space consists of elements of $\mcal{O}(M_n)$ where $\End_n$ acts trivially. Let $\mcal{O}(M_n) = k[X_{ij}, X_{ij}']$, where $X_{ij}$ corresponds to the $ij$th matrix entry of $X \in \Mat_n$, with coproduct
    \begin{equation*}
        \Delta(X_{ij}) = X_{ij} \otimes 1 + 1 \otimes X_{ij} + \sum_{k = 1}^n X_{ik}' \otimes X_{kj}.
    \end{equation*}
    Elements in $\End_n$ are their own inverses, so the $\End_n$-action of $Y \in \End_n(A)$ takes a matrix $X$ to $X + Y + X'Y$ and $X'$ to $X'$. Thus
    \begin{equation*}
        (\1 \otimes \mcal{O}(M_n))^{\End_n} \cong k[X_{ij}']
    \end{equation*}
    and we can rewrite $k[X_{ij}']$ as $k[1 + X_{ii}', X_{ij}']$ with $i \ne j$. Notice that the $M_n$-action of $Y \in M_n(A)$ by left multiplication on this representation sends $1 + X'$ to $(1 + Y')(1 + X')$. Hence the irreducible $M_n$-subrepresentations of $\ind_{\End_n}^{M_n} \1$ are exactly the polynomial $\GL(n)$-representations: given such a $\GL(n)$-representation $V$, the corresponding $M_n$-representation is the pullback along the map $M_n \to \GL(n)$ sending $X \mapsto I + X'$ where $I$ is the identity matrix.
\end{proof}

\begin{corollary}
    Let $L_H(\lambda_1, \dots, \lambda_n)$ be the irreducible $H_n$-representation corresponding to the $n$-tuple of integers $(\lambda_1, \dots, \lambda_n)$. Then the irreducible $\GL(nP)$-representation with highest weight $(T_{\lambda_1}, \dots, T_{\lambda_n})$, with $T_{\lambda} = \chi^{\lambda/4}$ when $\lambda \equiv 0 \pmod{4}$, is a submodule of \newline $\ind_{H_n}^{\GL(nP)} L_H(\lambda_1, \dots, \lambda_n)$. 
    \end{corollary}
    \begin{corollary}
    Let $L_M(\lambda_1, \dots, \lambda_n)$ be the irreducible $M_n$-representation corresponding to $(\lambda_1, \dots, \lambda_n)$. The $\GL(nP)$-representation with highest weight $(\xi^{\lambda_1}, \dots, \xi^{\lambda_n})$ is a submodule of $\ind_{M_n}^{\GL(nP)} L_M(\lambda_1, \dots, \lambda_n)$.
\end{corollary}
\begin{remark}
We can explicitly construct the $\GL(nP)$-module with highest weight \newline $(T_{\lambda_1}, \dots, T_{\lambda_n})$ by finding an invariant function $\varphi \in \Hom_H(L_H(\lambda_1, \dots, \lambda_n), \mcal{O}(\GL(nP)))$ and considering the $\GL(nP)$-module generated by the image of $\varphi$.
Let $v \in L:= L_H(\lambda_1, \dots, \lambda_n)$ be a highest weight vector for the $(\mathbb G'_m)^n$ action. Define the polynomials
\begin{equation*}
    p_{\lambda_i}(x) = \begin{cases} x^{\lambda_i} & \lambda_i \equiv 0, 1 \pmod{4} \\
    x^{\lambda_i} + x^{\lambda_i - 1}x' & \lambda_i \equiv 2, 3 \pmod{4} \end{cases},
\end{equation*}
so an upper triangular matrix in $H_n$ with $(a_{11}, \dots, a_{nn})$ on the diagonal acts on $v$ by the eigenvalue $\prod_i p_{\lambda_i}(a_{ii})$. Because the $H_n$-action on $v$ generates $L$ and $\varphi$ is $H_n$-invariant, $\varphi$ is defined by $\varphi(v)$. Extending $v$ to a homogeneous basis of $L$, let $v^*$ be the dual basis element corresponding to $v$ in $L^*$. Recall that we have an identification 
\begin{equation*}
    \Hom_H(L, \mcal{O}(\GL(nP))) \cong L^* \otimes \mcal{O}(M_n).
\end{equation*}
We pick $\varphi$ to be the function corresponding to $v^* \otimes 1$. Hence for $w \in L$, $\varphi(w) \in \mcal{O}(GL(nP))$ sends $h$ to the $v$-coefficient of $h(w)$.
In particular,
\begin{equation*}
    \varphi(v) = F := \prod_i p_{\lambda_i}(A_{ii}) + \text{ terms involving } A_{ij}, i > j
\end{equation*}
and $F$ has degree $\sum_i \lambda_i$ as a polynomial in $A_{ij}, A_{ij}'$.

Then consider the $\GL(nP)$-representation $L$ generated by $F$ and let us determine its highest $\GL(P)^n$-weight. 
Let $(a_i, b_i)_{1 \le i \le n}$ be an element of $\GL(P)^n$. This element acts by
\begin{align*}
    A_{ij} &\mapsto A_{ij} \otimes a_j + A'_{ij} \otimes b_j, \\
    B_{ij} &\mapsto B_{ij} \otimes a_j + (A_{ij} + B'_{ij}) \otimes b_j,
\end{align*}
so the $\GL(P)^n$-action on $F$ is
\begin{equation*}
    F \mapsto \prod_i p_{\lambda_i}(A_{ij} \otimes a_j + A'_{ij} \otimes b_j) + \text{ terms with lower } a_1 \text { degree}.
\end{equation*}

Since $F$ has the highest $A_{11}$ degree among elements in $L$, it also has the highest $a_1$ degree for the $\GL(P)^n$-action. So $F$ generates the $\GL(nP)$-representation with highest weight $(T_{\lambda_1}, \dots, T_{\lambda_n})$.
\end{remark}
\begin{remark}
The $\GL(nP)$-module with highest weight $(\xi^{\lambda_1}, \dots, \xi^{\lambda_n})$ can be explicitly constructed by finding an invariant function in $\Hom(L_M(\lambda_1, \dots, \lambda_n), \mcal{O}(\GL(nP)))$ and considering the $\GL(nP)$-module generated by its image.
    The construction is the same as above but the the invariant function $F$ will be $\prod ((1 + (BA^{-1})')_{ii})^{\lambda_{i}}$, so the representation will have highest $\GL(P)^n$-weight $(\xi^{\lambda_1}, \dots, \xi^{\lambda_n})$.
\end{remark}

\subsubsection{The Steinberg tensor product theorem for $\GL(m + nP)$}

While the Frobenius map $a \mapsto a^2$ is not a homomorphism in $\Ver_4^+$, it is true that $a \mapsto a^4$ is a homomorphism. In particular, this induces a homomorphism $\Fr_{\text{odd}}: \GL(m + nP) \to \GL(n)$, acting on an element of $\GL(m+nP)(A)$ as
\begin{equation*}
    \begin{pmatrix}
        F & C & C' \\
        B' & D & D' \\
        B & E & D+E'
    \end{pmatrix} \mapsto D^{(2)}
\end{equation*}
where $D^{(2)}$ denotes the matrix in $\GL(n)(A)$ with entries $d_{ij}^4$. 

We also have a homomorphism $\Fr_{\text{even}}: \GL(m + nP) \to \GL(m)$ which sends an element of $\GL(m + nP)(A)$ as above to $F^{(1)}$ (which is the matrix with entries $f_{ij}^2$). 
\begin{proposition}
$\Fr_{\text{even}}$ is a homomorphism.
\end{proposition}
\begin{proof}
Let us compute the $ij$th matrix entry of
\begin{equation*}
    \Fr_{\text{even}}\left(\begin{pmatrix}
        F_1 & C_1 & C_1' \\
        B_1' & D_1 & D_1' \\
        B_1 & E_1 & D_1+E_1'
    \end{pmatrix} \begin{pmatrix}
        F_2 & C_2 & C_2' \\
        B_2' & D_2 & D_2' \\
        B_2 & E_2 & D_2+E_2'
    \end{pmatrix}\right) = (F_1F_2 + C_1 B_2' + C_1' B_2)^{(1)}.
\end{equation*}
Denote the matrix entries of $F_k$ as $f^k_{ij}$ (all of which are in the kernel of $d$), the matrix entries of $C_1$ as $c_{ij}$, and the matrix entries of $B_2$ as $b_{ij}$. Then 
\begin{equation*}
    (F_1F_2 + C_1 B_2' + C_1' B_2)^{(1)}_{ij} = \left(\sum_k f^1_{ik} f^2_{kj} + c_{ik} b_{kj}' + c'_{ik} b_{kj}\right)^2.
\end{equation*}
We have that each $f^1_{ik} f^2_{kj}$, being in $\ker d$, commutes with all other terms, while each $c_{ik} b_{kj}'$ and $c_{ik}' b_{kj}$ squares to 0. Note also that
\begin{equation*}
    c_{ik}b'_{kj}c_{ik'}b'_{k'j} + c_{ik'}b'_{k'j}c_{ik}b'_{kj} = c'_{ik'}c_{ik}'b'_{kj}b'_{k'j} = c'_{ik}b_{kj}c'_{ik'}b_{k'j} + c'_{ik'}b_{k'j} + c'_{ik}b_{kj}.
\end{equation*}
Likewise,
\begin{equation*}
    c_{ik}b'_{kj}c'_{ik'}b_{k'j} + c'_{ik'}b_{k'j}c_{ik}b'_{kj} = c'_{ik}b_{kj}c_{ik'}b'_{k'j} + c_{ik'}b'_{k'j}c'_{ik}b_{kj}.
\end{equation*}
Hence, all terms except those in $(F_1F_2)^{(1)}$ cancel, so
\begin{equation*}
    (F_1F_2 + C_1 B_2' + C_1' B_2)^{(1)} = (F_1 F_2)^{(1)}.
\end{equation*}
\end{proof}

Because both of these maps correspond to injective maps on their coordinate rings ($\Fr_{\text{odd}}$ sends matrix functions to their fourth powers, and $\Fr_{\text{even}}$ to their squares), they are surjections of schemes. Therefore, the pullback of an irreducible $\GL(m)$ or $\GL(n)$-module along one of these homomorphisms is an irreducible $\GL(m + nP)$-module. By checking highest weights, we have
\begin{align*}
    \Fr_{\text{even}}^*(L_{\GL(m)}(\lambda)) &= L(2\lambda, \1) \\
    \Fr_{\text{odd}}^* (L_{\GL(n)}(\lambda)) &= L(0, \chi^\mu)
\end{align*}
where $\chi^\mu = (\chi^{\mu_1}, \dots, \chi^{\mu_n})$.

\begin{remark}
    The terminology of an even and odd Frobenius comes from the following: for an object $X = V \oplus W \otimes P \in \Ver_4^+$, notice that $\Fr_{\text{odd}}$ corresponds to the functor (which is not a tensor functor) $X \mapsto W^{(2)}$, where $W^{(2)}$ is the second Frobenius twist of $W$, while $\Fr_{\text{even}}$ corresponds to the functor $X \mapsto V^{(1)} = X^{(1)}$, the usual Frobenius twist of $X$.
\end{remark}

We expect that a Steinberg tensor product theorem should hold for $\GL(m + nP)$ using both $\Fr_{\text{odd}}$ and $\Fr_{\text{even}}$. Say that a highest weight $\Lambda$ of $\GL(nP)$ is $4$-restricted if $\deg \Lambda_i - \deg \Lambda_{i + 1} < 4$. Then every highest weight of $\GL(nP)$ can be written as $\Lambda \cdot \chi^\mu = (\chi^{\mu_1} \otimes \Lambda_1, \dots, \chi^{\mu_n} \otimes \Lambda_n)$ where $\Lambda$ is $4$-restricted and $\mu = (\mu_1, \dots, \mu_n)$ is a highest $\GL(n)$ weight. Say that a highest weight $(\mu, \Lambda)$ of $\GL(m + nP)$ is restricted if $\mu$ is 2-restricted and $\Lambda$ is 4-restricted as a $\GL(nP)$-highest weight. Our conjecture states that every irreducible representation of $\GL(m + nP)$ can be written as a tensor product of $L(\mu, \Lambda)$ for $(\mu, \Lambda)$ restricted with Frobenius pullbacks of irreducible $\GL(m)$ and $\GL(n)$ modules.

\begin{conjecture}
    For $\lambda$ a 4-restricted $\GL(nP)$ weight, $\mu$ a 2-restricted $\GL(m)$-weight, $\nu$ a $\GL(n)$-weight, and $\rho$ a $\GL(m)$-weight, we have
\label{steinberg}
    \begin{equation*}
        L_{\GL(m + nP)}(\mu + 2\rho,  \lambda \cdot \chi^\nu) \cong L_{\GL(m + nP)}(\mu, \lambda) \otimes (\Fr_{\text{odd}})^* L_{\GL(n)}(\nu) \otimes (\Fr_{\text{even}})^* L_{\GL(m)}(\rho).
    \end{equation*}
\end{conjecture}
A Steinberg tensor product theorem for $\sVec$ was proved in \cite{kujawa_steinberg_2006} and one for $\Ver_p$ was proved in \cite{kannan_steinberg_2024}; we expect to prove Conjecture \ref{steinberg} in $\Ver_4^+$ with similar methods.

\printbibliography

\end{document}